\documentclass[reqno]{amsart}

\newtheorem{theorem}{Theorem}[section]
\newtheorem{lemma}[theorem]{Lemma}
\newtheorem{corollary}[theorem]{Corollary}
\newtheorem{proposition}[theorem]{Proposition}

\theoremstyle{definition}

\usepackage[T1]{fontenc}
\usepackage[utf8]{inputenc}
\usepackage{lmodern}
\usepackage{textcomp}
\usepackage{microtype}
\usepackage{tikz}
\usepackage{mathtools, bm, nccmath}

\newcounter{proofstep}

\usepackage[english]{babel}
\usepackage{csquotes}
\usepackage[export]{adjustbox}
\usepackage{circuitikz}

\usepackage{tabularx}

\usepackage{graphicx}
\usepackage{setspace}

\usepackage{amsmath}
\usepackage{amsfonts}
\usepackage{amssymb}
\usepackage{amsthm}
\usepackage{xcolor}%
\usepackage{soul}
\usepackage{graphicx} 
\usepackage{lipsum}
\numberwithin{equation}{section}

\newcommand*{\R}{\mathbb{R}}
\newcommand*{\C}{\mathbb{C}}

\usepackage{thmtools}

\declaretheorem[
	name=Remark,
	style=remark,
	numbered=no
	]{rem}

\newcounter{thmletter}

\usepackage[pdftex,pdfpagelabels]{hyperref}
\numberwithin{equation}{section}

\newcommand\restr[2]{{
  \left.\kern-\nulldelimiterspace 
  #1 
  \littletaller 
  \right|_{#2} 
  }}

\newcommand{\littletaller}{\mathchoice{\vphantom{\big|}}{}{}{}}

\makeatletter
\@namedef{subjclassname@2020}{\textup{2020} Mathematics Subject Classification}
\makeatother

\calclayout

\begin{document}
\title[Eigenvalue bounds for generalized Schr\"odinger operators]
{Resolvent bounds and eigenvalue estimates of generalized Schrödinger operators with complex potentials on compact manifolds}

\author[E. Stefanescu]{Eduard Stefanescu}
\address{Institut f\"ur Analysis und Zahlentheorie, TU Graz, Steyrergasse 30, 8010 Graz, Austria}
\email{\href{mailto:eduard.stefanescu@tugraz.at}{eduard.stefanescu@tugraz.at}}


\begin{abstract}

We extend Cuenin's compact-manifold spectral bounds for Schrödinger operators
with complex potentials to a general pseudodifferential setting. More precisely,
we study operators \(P+V\), where \(P\) is a positive self-adjoint elliptic
classical pseudodifferential operator of positive order and \(V\) is
complex-valued.

The main analytic input is a resolvent principle showing that spectral cluster
estimates for \(P\) imply \(L^p\)-\(L^{p'}\) resolvent estimates along suitable
complex curves. Combined with Sogge's spectral cluster bounds, this yields
exterior-region resolvent estimates extending those of Krupchyk and Uhlmann;
we also prove direct resolvent bounds in the interior region. On Zoll
manifolds, we discuss the sharpness of the resulting spectral bounds.
\end{abstract}

\maketitle




\section{Introduction}

Uniform resolvent estimates are a fundamental tool in the spectral theory of
Schrödinger operators and in the study of eigenvalue bounds for non-self-adjoint
perturbations. On compact Riemannian manifolds, resolvent estimates for the
Laplace--Beltrami operator were established by Dos Santos Ferreira, Kenig and
Salo \cite{DosSantosFerreiraKenigSalo2014}. More precisely, they proved
\(L^p\)-resolvent estimates for \((-\Delta_g-z)^{-1}\) in the Sobolev-critical
range, with the spectral parameter \(z\) restricted to the exterior of a
parabolic neighbourhood of the positive real axis. The endpoint case was later
obtained by Frank and Schimmer \cite{FRANKSchimmer2017}. For higher order
elliptic self-adjoint differential operators on compact manifolds,
corresponding estimates were proved by Krupchyk and Uhlmann \cite{KU15} under a
suitable curvature assumption on the characteristic hypersurfaces.

These results are part of a broader circle of uniform Sobolev and resolvent
estimates going back to Kenig, Ruiz and Sogge
\cite{KenigRuizSogge1987}. Further developments include the work of
Bourgain, Shao, Sogge and Yao \cite{BourgainShaoSoggeYao2015}, which refines
the connection between spectral cluster estimates and resolvent bounds on
compact manifolds. Related uniform Sobolev estimates in the periodic setting
were obtained by Shen and Zhao \cite{ShenZhao08}.

A different and conceptually useful point of view was introduced by Cuenin in
\cite{CUENIN2024110214}. He showed that, in several relevant situations,
uniform resolvent estimates can be derived directly from spectral cluster
estimates. In particular, Sogge's spectral cluster bounds for the
Laplace--Beltrami operator imply the corresponding uniform resolvent estimates
without relying on the explicit construction of a Hadamard parametrix. This
functional-analytic approach is especially well suited for extensions to
operators for which suitable spectral cluster estimates are already available.

The aim of the present paper is to extend this circle of ideas to a class of
self-adjoint elliptic classical pseudodifferential operators. Let \(P\) be a
positive self-adjoint elliptic classical pseudodifferential operator of order
\(m>0\) on a closed \(d\)-dimensional manifold \(M\). We assume that the
principal symbol \(p_m(x,\xi)\) is positive for \(\xi\neq 0\), homogeneous of
degree \(m\) in \(\xi\), and that the associated cospheres
\[
    \Sigma_x
    :=
    \{\xi\in T_x^*M : p_m(x,\xi)=1\}
\]
have everywhere non-vanishing Gaussian curvature. Equivalently, after passing
to the first-order operator \(P^{1/m}\), this is the curvature condition under
which spectral cluster estimates for functions of pseudodifferential operators
apply; see \cite[Chapter 3-5]{CS} or \cite{Sogge1986}.

Our first objective is to prove \(L^p\)-resolvent estimates for
\[
    (P-z)^{-1}
\]
in regions of the complex plane adapted to the order \(m\). The initial
estimates are obtained on a contour of the form
\[
    \Gamma
    =
    \{(\lambda+i)^m:\lambda\geq \cot(\pi/m)\}
    \cup
    \{(\lambda-i)^m:\lambda\geq \cot(\pi/m)\},
\]
where \(\lambda\) is the spectral parameter of \(P^{1/m}\). The bounds on
\(\Gamma\) are derived from spectral cluster estimates and Sobolev embeddings.
They are then extended to the exterior of the contour by a
Phragmén--Lindelöf argument. Near the singularities, that is, close to the
spectrum of \(P\), the bounds are obtained by combining spectral cluster
estimates with elementary estimates on the geometry of the complex parameter.
This strategy is strongly inspired by Cuenin's works \cite{CJC} and \cite{CUENIN2024110214}.

As an application, we study non-self-adjoint Schrödinger-type operators
\[
    P+V,
\]
where \(V\) is a complex-valued potential. Using the resolvent estimates for
\(P\) together with the Birman--Schwinger principle, we prove spectral
inclusion bounds for the eigenvalues of \(P+V\). In particular, the spectrum
is shown to lie in a union of complex discs centred at the eigenvalues of the
unperturbed operator \(P\). The radii of these discs are controlled by
Lebesgue norms of \(V\) and by the spectral cluster exponent associated with
\(P\).

We also discuss sharpness. In the model case of the Laplace--Beltrami operator,
and more generally for operators with Zoll-type spectral clustering, the high
multiplicity and clustering structure of the spectrum leads to examples showing
that the radii of the spectral inclusion discs cannot, in general, be improved.

The eigenvalue bounds proved here continue a line of work on Schrödinger
operators with real and complex potentials. In the self-adjoint case, such
questions go back at least to Keller's eigenvalue estimates \cite{K61}. For
complex potentials, early bounds for non-real eigenvalues were obtained by
Abramov, Aslanyan and Davies \cite{AAD}. A fundamental result in the Euclidean
setting is Frank's eigenvalue bound for Schrödinger operators with complex
potentials \cite{FR1}, while Cuenin later proved a compact-manifold analogue
for the Laplace--Beltrami operator \cite{CJC}. The present paper extends this
compact-manifold framework to positive self-adjoint elliptic classical
pseudodifferential operators satisfying the above curvature condition. For a
recent survey by the author of deterministic and random eigenvalue bounds for
non-self-adjoint Schrödinger operators, see \cite{StefanescuSurvey}.

In the Euclidean setting, Laptev and Safronov conjectured that such eigenvalue
bounds should hold in a larger range of exponents. This conjecture was disproved
by Frank and Simon for embedded positive eigenvalues \cite{FRS2}, and by
Bögli and Cuenin in the remaining complex-eigenvalue regime \cite{CJCBS}. Thus
Franks condition is sharp in general. 

Further related developments include improved eigenvalue bounds for sparse complex potentials \cite{CueninSparse}, random Schrödinger operators with complex potentials \cite{CJCMK,3KMJCCKyoto,CMS}.

\subsection{Structure of the paper}

The main results are stated in Section~\ref{sec:main-results}. Section~\ref{sec:preliminaries} contains the notation and
preliminary material needed throughout the paper. The boundary resolvent estimates are proved in Section~\ref{sec6}. The
Phragmén--Lindelöf extension and the bounds near the spectrum are proved in
Sections~\ref{sec: complex ext} and~\ref{sec: boundsnearsing}, respectively. In Section~\ref{sec7}, we use these resolvent estimates
together with the Birman--Schwinger principle to prove the spectral inclusion
results. Finally, Sections~\ref{secopt} and \eqref{sec: optimality laplace} is devoted to the sharpness of the obtained bounds
on the sphere and, more generally, on Zoll manifolds.

\section{Main results}\label{sec:main-results}

\subsection*{Standing assumptions}

Throughout the main results, \(M\) is a smooth closed \(d\)-dimensional Riemannian
manifold, \(d\geq 2\), and
\[
    P\in\Psi_{\mathrm{cl}}^m(M),
    \qquad
    \frac{2d}{d+1}\leq m\leq d,
\]
is a positive self-adjoint elliptic classical pseudodifferential operator. Its
principal symbol is denoted by
\[
    p_m(x,\xi):=\sigma_m(P)(x,\xi),
\]
and we assume
\[
    p_m(x,\xi)>0,
    \qquad
    (x,\xi)\in T^*M\setminus 0 .
\]
For each \(x\in M\), define the cosphere
\begin{equation}\label{eq: cosphere}
    \Sigma_x:=\{\xi\in T_x^*M:p_m(x,\xi)=1\}.
\end{equation}
We assume that \(\Sigma_x\) has everywhere non-vanishing Gaussian curvature.

The exponents \(p,p'\) are
Hölder conjugates,
\[
    \frac1p+\frac1{p'}=1,
\]
and we assume that \(p'\) lies in the Sobolev-admissible range
\begin{equation}\label{eq: admissible}
    2\le p'< \frac{2d}{d-m},
    \qquad
    \text{if } m<d,
\end{equation}
and
\[
    2 \le p'<\infty,
    \qquad
    \text{if } m=d.
\]
For \(2\leq p'\leq\infty\), we define
\begin{equation}\label{eq: nu}
    \nu(p')
    :=
    \begin{cases}
    \displaystyle
    \frac{d-1}{2}\left(\frac12-\frac1{p'}\right),
    &
    2\leq p'\leq \frac{2(d+1)}{d-1},
    \\[1.4ex]
    \displaystyle
    d\left(\frac12-\frac1{p'}\right)-\frac12,
    &
    \frac{2(d+1)}{d-1}\leq p'\leq\infty.
    \end{cases}
\end{equation}
For the eigenvalue bounds with potentials \(V\in L^q(M)\), the exponent \(q\)
is related to \(p,p'\) by
\[
    \frac1q=\frac1p-\frac1{p'}.
\]
The admissibility condition on \(p'\) is equivalent to
\[
    \frac dm<q\leq \infty,
    \qquad
    \text{if } m<d,
\]
and
\[
    1<q\leq\infty,
    \qquad
    \text{if } m=d.
\]
For \(m=d\), this condition reads \(q>1\), equivalently
\(p'<\infty\).

In this parametrization we set
\[
    \sigma(q):=\nu\!\left(\frac{2q}{q-1}\right)=\nu(p').
\]
Equivalently,
\[
    \sigma(q)
    =
    \begin{cases}
    \displaystyle
    \frac{d}{2q}-\frac12,
    &
    \displaystyle
    \frac dm< q\leq \frac{d+1}{2},
    \\[1.4ex]
    \displaystyle
    \frac{d-1}{4q},
    &
    \displaystyle
    \frac{d+1}{2}\leq q\leq\infty.
    \end{cases}
\]

For notational convenience, we include the point \(0\) among the centres in the
spectral inclusion theorem. Thus, in the statement below, we use the convention $\lambda_0:=0$, even if \(0\notin\operatorname{spec}(P)\). For \(k\geq 1\), the numbers
\(\lambda_k^m\) denote the non-zero eigenvalues of \(P\), counted with
multiplicity and arranged in non-decreasing order. If \(0\) is an eigenvalue,
this convention merely suppresses possible repetitions of the centre \(0\);
if \(0\notin\operatorname{spec}(P)\), it adds one harmless auxiliary disc
centred at the origin.

Our main spectral inclusion result is the following.

\begin{theorem}\label{psithm1}
Assume the standing assumptions. Then there exists a constant
\(C=C(M,g,P,q,m)>0\) such that, for every \(V\in L^q(M)\),
\begin{equation}\label{result}
    \operatorname{spec}(P+V)
    \subset
    \bigcup_{k=0}^\infty
    D\bigl(\lambda_k^m, C r_k\bigr)
    \;\cup\;
    \left\{
        z\in\mathbb C :
        |z|^{1-\frac1m}
        (1+|z|)^{-\frac{2\sigma(q)}{m}}
        \leq C \|V\|_{L^q(M)}
    \right\},
\end{equation}
where
\[
    r_k
    :=
    \|V\|_{L^q(M)}(1+\lambda_k)^{2\sigma(q)}.
\]
\end{theorem}

The sharpness argument requires more than the curvature assumption on the
principal cospheres. In addition, the spectrum of the first-order operator
must exhibit sufficiently narrow clusters, and the corresponding spectral
projectors must saturate the spectral-cluster estimates. We formulate the
argument under abstract assumptions which include both the round sphere and
fractional powers of the Laplace--Beltrami operator on Zoll manifolds.

\begin{theorem}[Sharpness under Zoll-type spectral clustering]
\label{thmopt}
Assume the standing assumptions and let
\[
    A:=P^{1/m}.
\]
Suppose that there exist numbers
$\Lambda_k\asymp k$ and $\delta_k\ge 0$, such that the intervals
\[
    I_k:=
    \left[\Lambda_k-\frac12,\Lambda_k+\frac12\right)
\]
are mutually disjoint for all sufficiently large \(k\), and such that
\begin{equation}
\label{eq:zoll-type-clustering}
    \operatorname{spec}(A)\cap I_k
    \subset
    [\Lambda_k-\delta_k,\Lambda_k+\delta_k].
\end{equation}
Let $\Pi_k:=\boldsymbol{1}_{I_k}(A)$ be the corresponding spectral projector.
Assume, moreover, that
\begin{equation}
\label{eq:sharp-cluster-lower}
    \|\Pi_k\|_{L^p(M)\to L^{p'}(M)}
    \gtrsim
    k^{2\sigma(q)}.
\end{equation}
Let \(\theta\in[0,2\pi)\), and let
\(\eta_k>0\) satisfy
\begin{equation}
\label{eq:sharpness-smallness}
    \varepsilon_k
    :=
    \eta_k k^{2\sigma(q)+1-m}
    +
    \frac{k^{m-1}\delta_k}
         {\eta_k k^{2\sigma(q)}}
    \longrightarrow0.
\end{equation}
Then, for every sufficiently large \(k\), there exists a potential
\[
    V_k\in L^q(M),
    \qquad
    \|V_k\|_{L^q(M)}\asymp\eta_k,
\]
such that \(P+V_k\) has an eigenvalue \(z_k\) satisfying
\begin{equation}
\label{eq:abstract-sharpness-conclusion}
    z_k
    =
    \Lambda_k^m
    +
    \eta_k e^{i\theta}k^{2\sigma(q)}
    +
    O\left(
        \eta_k^2 k^{4\sigma(q)+1-m}
        +
        k^{m-1}\delta_k
    \right).
\end{equation}
In particular,
\[
    z_k
    =
    \Lambda_k^m
    +
    \eta_k e^{i\theta}
    \bigl(1+o(1)\bigr)k^{2\sigma(q)}.
\]

For \(q=\infty\), the sharpness statement holds trivially by taking a
constant potential.
\end{theorem}

\begin{corollary}[Fractional Laplacian on the round sphere]
\label{cor:sharpness-sphere}
Let
\[
    M=\mathbb S^d,
    \qquad
    P_m=(-\Delta_{\mathbb S^d})^{m/2}.
\]
For every \(\eta>0\), \(\theta\in[0,2\pi)\), and
\[
    k\gg
    1+\eta^{1/(m-1-2\sigma(q))},
\]
there exists
\[
    V_{k,\theta,\eta}\in L^q(\mathbb S^d),
    \qquad
    \|V_{k,\theta,\eta}\|_{L^q(\mathbb S^d)}
    \asymp\eta,
\]
such that
\begin{align}
\label{eq:sphere-sharpness-final}
    &\bigl(k(k+d-1)\bigr)^{m/2}
    +
    \eta e^{i\theta}
    \left(
        1+
        O\bigl(
            \eta k^{2\sigma(q)+1-m}
        \bigr)
    \right)
    k^{2\sigma(q)}
    \notag\\
    &\hspace{4cm}
    \in
    \operatorname{spec}
    \bigl(P_m+V_{k,\theta,\eta}\bigr).
\end{align}
Consequently, the radii
\[
    \|V\|_{L^q(\mathbb S^d)}
    k^{2\sigma(q)}
\]
in the spectral inclusion theorem are optimal up to multiplicative constants.
\end{corollary}

\begin{corollary}[Fractional Laplacian on a Zoll manifold]
\label{cor:sharpness-zoll}
Let \((M,g)\) be a Zoll manifold whose unit-speed geodesics have common
minimal period \(2\pi\), and let
\[
    P_m=(-\Delta_g)^{m/2}.
\]
Let \(\eta>0\) and suppose that
\begin{equation}
\label{eq:zoll-two-conditions}
    \eta k^{2\sigma(q)+1-m}\ll1,
    \qquad
    \eta k^{2\sigma(q)+2-m}\gg1.
\end{equation}
Then, for every \(\theta\in[0,2\pi)\), there exists
\[
    V_{k,\theta,\eta}\in L^q(M),
    \qquad
    \|V_{k,\theta,\eta}\|_{L^q(M)}
    \asymp\eta,
\]
such that \(P_m+V_{k,\theta,\eta}\) has an eigenvalue
\begin{equation}
\label{eq:zoll-sharpness-final}
    z_k
    =
    (k+\alpha)^m
    +
    \eta e^{i\theta}k^{2\sigma(q)}
    +
    O\left(
        \eta^2 k^{4\sigma(q)+1-m}\right)
        +
        k^{m-2}=
    (k+\alpha)^m
    +
    \eta e^{i\theta}
    \bigl(1+o(1)\bigr)k^{2\sigma(q)},
\end{equation}
where by Weinstein  \cite{Weinstein1977Clusters} spec$(\sqrt{-\Delta_g})\ \cap I_k \subset [k+\alpha-C/k,k+\alpha+C/k]$, for some $\alpha\in\R$.
\end{corollary}

\begin{rem}[More general operators with periodic bicharacteristics]
Theorem~\ref{thmopt} is not restricted to powers of the Laplace--Beltrami
operator. Positive elliptic operators with periodic bicharacteristic flow were studied
by Duistermaat and Guillemin
\cite{DuistermaatGuillemin1975Periodic} and by Colin de Verdière
\cite{ColinDeVerdiere1979Periodic}. Periodicity of the bicharacteristic flow
alone, however, should not be substituted for
\eqref{eq:zoll-type-clustering}: lower-order and subprincipal terms may affect
the location and width of the spectral clusters. The abstract spectral
hypotheses above isolate exactly the properties needed in the sharpness
argument.
\end{rem}

The proof of Theorem~\ref{psithm1} is based on resolvent estimates for the
unperturbed operator \(P\). More precisely, we first prove uniform
\(L^p\)-resolvent bounds along a complex contour adapted to the order \(m\) of
the operator. These estimates generalize the higher-order differential
operator bounds of Cuenin~\cite{CUENIN2024110214} and Krupchyk--Uhlmann~\cite{KU15}
to the present pseudodifferential setting. Following Cuenin's strategy, the
main point is that suitable spectral cluster estimates are sufficient to
obtain the required resolvent bounds.

Define
\[
\Gamma_+
:=
\left\{
(\lambda+i)^m : \lambda\geq \cot\!\left(\frac{\pi}{m}\right)
\right\},
\qquad
\Gamma_-
:=
\left\{
(\lambda-i)^m : \lambda\geq \cot\!\left(\frac{\pi}{m}\right)
\right\},
\]
and set
\[
    \Gamma:=\Gamma_+\cup\Gamma_-.
\]
The two arcs meet at the point
\begin{equation}\label{Gammacontinuous}
    z_\ast
    :=
    \left(\cot\!\left(\frac{\pi}{m}\right)+i\right)^m
    =
    \left(\cot\!\left(\frac{\pi}{m}\right)-i\right)^m
    =
    -\sin\!\left(\frac{\pi}{m}\right)^{-m}.
\end{equation}
    
Thus \(\Gamma\) is a continuous curve in \(\mathbb C\). We define
\(\Xi_0\) to be the connected component of \(\mathbb C\setminus\Gamma\)
which contains the half-line $(-\infty,z_\ast)$.
Finally, we set
\[
    \Xi:=\Xi_0\cup\Gamma,
\]
see Figure~\ref{fig}.

\begin{figure}
    \centering
    \includegraphics[width=0.75\linewidth]{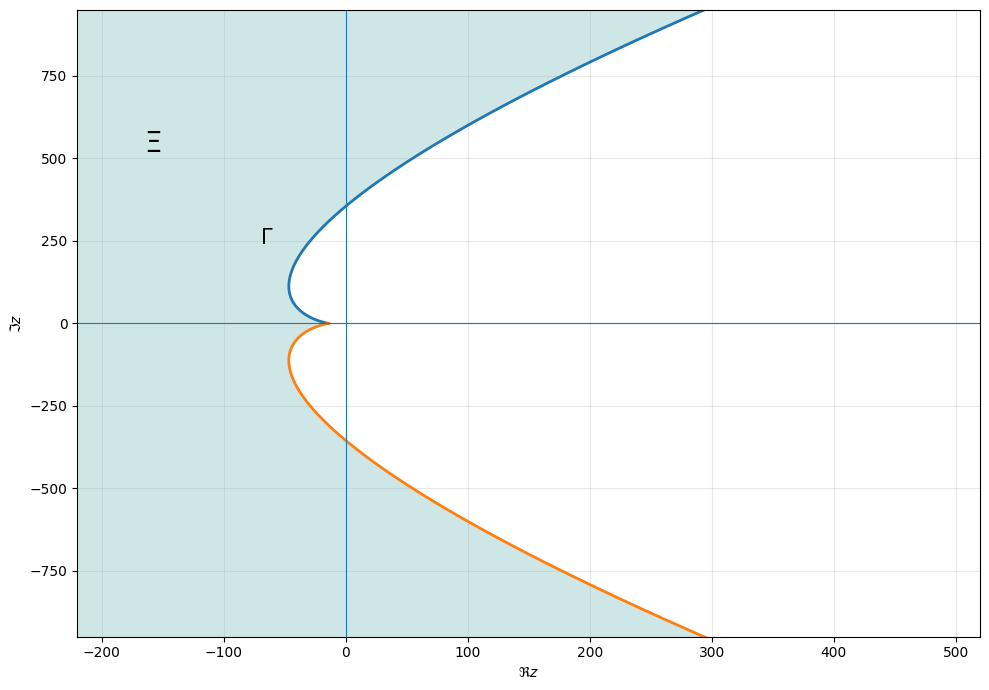}
    \caption{The region \(\Xi\) and its boundary \(\Gamma\) for \(m=5\).}
    \label{fig}
\end{figure}

\begin{theorem}\label{eqithm-regionuation}
Assume the standing assumptions. Then the following resolvent estimates hold.

First, for every sufficiently large \(\lambda\in\mathbb R\),
\begin{equation}\label{eqithm-boundary}
    \left\|
        \bigl(P-(\lambda\pm i)^m\bigr)^{-1}
    \right\|_{L^p(M)\to L^{p'}(M)}
    \lesssim
    \langle \lambda\rangle^{2\nu(p')+1-m},
\end{equation}
where $\nu(p')$ is as in \eqref{eq: nu}.

Moreover, this estimate propagates to the whole region \(\Xi\). More precisely,
for every \(z\in\Xi\),
\begin{equation}\label{eqithm-region}
    \left\|
        (P-z)^{-1}
    \right\|_{L^p(M)\to L^{p'}(M)}
    \lesssim
    |z|^{\frac{2\nu(p')+1}{m}-1},
\end{equation}
uniformly in \(z\).

Finally, for $z\in\C\backslash\{\Xi\cup[0,\infty)\}$, one has the singular resolvent estimate
\begin{equation}\label{secondequation}
    \left\|
        (P-z)^{-1}
    \right\|_{L^{p}(M)\to L^{p'}(M)}
    \lesssim
    d(z)^{-1}(1+|z|)^{\frac{2\nu(p')}{m}},
\end{equation}
where \(d(z):=\operatorname{dist}(z,\operatorname{spec}(P))\).

\end{theorem}

\begin{rem}[Examples covered by the operator assumptions]
The operator assumptions are satisfied by the following classes. In every
case, the resulting order \(m\) is assumed to lie in the range
\[
    \frac{2d}{d+1}\leq m\leq d.
\]

\begin{itemize}
\item \emph{Fractional powers.}
Let
\[
    L\in\Psi_{\mathrm{cl}}^\ell(M),
    \qquad \ell>0,
\]
be a strictly positive self-adjoint elliptic classical pseudodifferential
operator whose principal symbol is positive away from the zero section and
satisfies the cosphere curvature condition \eqref{eq: cosphere}. Then
\[
    P=L^{m/\ell}
\]
is a positive self-adjoint elliptic classical pseudodifferential operator of
order \(m\) satisfying the same cosphere curvature condition. This includes,
in particular,
\[
    P=(-\Delta_g+c)^{m/2},
    \qquad c\ge0.
\]
The construction of complex and fractional powers of elliptic operators is due
to Seeley \cite{Seeley1967ComplexPowers}.

\item \emph{Positive weights and self-adjoint lower-order perturbations.}
Let \(P_0\in\Psi_{\mathrm{cl}}^m(M)\) satisfy the operator assumptions, let
\[
    b\in C^\infty(M;\mathbb R),
    \qquad b>0,
\]
and let
\[
    B\in\Psi_{\mathrm{cl}}^{m-1}(M),
    \qquad B=B^*,
    \qquad
    S\in\Psi^{-\infty}(M),
    \qquad S=S^*.
\]
Then, for \(C\) sufficiently large,
\[
    P=M_{\sqrt b}P_0M_{\sqrt b}+B+S+C
\]
is a positive self-adjoint elliptic classical pseudodifferential operator of
order \(m\), where \(M_{\sqrt b}\) denotes the multiplication operator by
\(\sqrt b\), that is, \(M_{\sqrt b}u(x)=\sqrt{b(x)}\,u(x)\). Its principal symbol is
\[
    \sigma_m(P)(x,\xi)=b(x)\sigma_m(P_0)(x,\xi),
\]
so its principal cospheres are positive fiberwise dilations of those of
\(P_0\), and the curvature condition is preserved.

This class includes, whenever the corresponding order lies in the admissible
range,
\[
    p\bigl((-\Delta_g+c)^{1/2}\bigr)+C,
\]
where \(p\in\mathbb R[t]\) has positive leading coefficient, as well as the
Laplace-type and magnetic Schrödinger operators
\[
    -\Delta_g+W+C
\]
and
\[
    (-i\,d+\mathcal A)^*(-i\,d+\mathcal A)+W+C,
\]
where
\[
    \mathcal A\in\Omega^1(M;\mathbb R),
    \qquad
    W\in C^\infty(M;\mathbb R).
\]
It also includes the Yamabe operator \cite{Yamabe1960}, the Paneitz operator
\cite{Paneitz2008}, and, more generally, the GJMS operators
\cite{GrahamJenneMasonSparling1992GJMS}, introduced by Graham, Jenne, Mason, and Sparling, after adding a sufficiently large
constant and provided that their order \(2k\) satisfies \(2k\leq d\).

\end{itemize}

Thus the assumptions include fractional powers, polynomial spectral functions,
Laplace-type and magnetic Schrödinger operators, conformally covariant
geometric operators, and natural self-adjoint perturbations and combinations
of such operators. The relation between the cosphere curvature condition and
spectral-cluster estimates is discussed, for example, by Sogge \cite{CS},
Seeger--Sogge \cite{SeegerSogge1989Eigenfunctions}, and
Krupchyk--Uhlmann \cite{KU15}.
\end{rem}

\section{Notation and Preliminaries}\label{sec:preliminaries}

Throughout, \(d\geq2\), \(\langle x\rangle=(1+|x|^2)^{1/2}\), and

\[
    D(c,r):=\{z\in\mathbb C:|z-c|\leq r\}.
\]

The notation \(X\lesssim Y\) means that \(X\leq CY\), with a constant
independent of the parameters under consideration; \(X\asymp Y\) means that
both inequalities hold. All \(L^p(M)\)-spaces are formed with respect to the
Riemannian volume measure. We identify a measurable function with its
multiplication operator whenever no confusion can arise.

We use the standard symbol classes \(S^s\), the Kohn--Nirenberg and Weyl
quantizations, and the classes \(\Psi^s\) and \(\Psi_{\mathrm{cl}}^s\) of
pseudodifferential and classical pseudodifferential operators. Thus
\(D_x=-i\partial_x\), \(a^w=\operatorname{Op}^w(a)\), and the principal symbol
of \(B\in\Psi_{\mathrm{cl}}^s(M)\) is denoted by \(\sigma_s(B)\). If $B\in\Psi^{-\infty}=\bigcap_{m\in\mathbb R}\Psi^m$, we call $B$ a smoothing operator. We use without
further comment the composition and adjoint formulae, elliptic parametrices,
the \(L^r\)-boundedness of order-zero operators for \(1<r<\infty\), and the
parameter-dependent functional calculus. These facts, as well as the local
symbol conventions used below, may be found in Stein
\cite[Chapter VI]{Stein93} and Zworski \cite[Chapters 4 and 14]{zw22}. All
local pseudodifferential operators are taken properly supported. In
particular, inserting two cutoffs with separated supports produces a smoothing operator.

Set

\[
    A:=P^{1/m}.
\]

The elliptic functional calculus shows that \(A\in\Psi_{\mathrm{cl}}^1(M)\)
is positive and self-adjoint, with principal symbol $a_1(x,\xi)=p_m(x,\xi)^{1/m}.$
Since \(P\) has compact resolvent, there is an orthonormal basis
\((e_j)_{j\geq0}\) such that

\[
    Ae_j=\lambda_j e_j,
    \qquad
    Pe_j=\lambda_j^m e_j,
    \qquad
    0\leq\lambda_0\leq\lambda_1\leq\cdots .
\]

For a Borel set \(J\subset[0,\infty)\), write $\Pi_J:=\mathbf 1_J(A).$
Thus, for \(z\notin\operatorname{spec}(P)\),

\[
    (P-z)^{-1}f
    =
    \sum_{j=0}^{\infty}
    \frac{\langle f,e_j\rangle}{\lambda_j^m-z}e_j
\]

in \(L^2(M)\). The curvature assumption gives Sogge's unit spectral-cluster
estimate \cite{CS}

\begin{equation}\label{Sogge1}
    \|\Pi_{[\mu,\mu+1]}\|_{L^2(M)\to L^{p'}(M)}
    \lesssim
    \langle\mu\rangle^{\nu(p')},
    \qquad \mu\geq0,
\end{equation}

and, by duality,

\begin{equation}\label{Sogge2}
    \|\Pi_{[\mu,\mu+1]}\|_{L^p(M)\to L^2(M)}
    \lesssim
    \langle\mu\rangle^{\nu(p')}.
\end{equation}

\section{
Resolvent bounds. Proof of
\texorpdfstring{Theorem~\ref{eqithm-regionuation},
equation~\eqref{eqithm-boundary}}
{the boundary estimate}
}
\label{sec6}

\begin{proof}We prove the estimate on the unbounded part of the contour, where
\(\lambda\geq\lambda_0\) for a fixed sufficiently large \(\lambda_0\). The
remaining compact part follows from the continuity of the resolvent on compact
subsets of the resolvent set. Write

\[
    \gamma:=\nu(p')
\]

throughout this section. We first isolate the Christ--Kiselev step, then apply it to
the first-order operator \(A-\lambda\), and afterwards rewrite the result
in terms of \(P-\lambda^m\).

\subsection{The Christ--Kiselev upgrade}

We first fix the phase-space notation used in the local argument. Write
\[
    x=(t,y)\in I\times\mathbb R^{d-1},
    \qquad
    \xi=(\tau,\eta)\in\mathbb R\times\mathbb R^{d-1},
\]
where \(I\subset\mathbb R\) is a bounded interval. After a linear change of
coordinates, a sufficiently small conic neighborhood of a fixed point of the
characteristic set, localized to frequencies \(|\xi|\asymp\lambda\), may be
placed in a region of the form
\[
    \mathcal B_\lambda
    :=
    \left\{
        (t,y;\tau,\eta):
        (t,y)\in Q,\quad
        c_0\lambda<\tau<C_0\lambda,\quad
        |\eta|<\varepsilon\tau
    \right\},
\]
where
\[
    Q\Subset I\times\mathbb R^{d-1},
    \qquad
    0<c_0<C_0<\infty,
\]
and \(\varepsilon>0\) are independent of \(\lambda\). This is the analogue of
the phase-space region \(B_\lambda\) in
\cite[Section~4]{CUENIN2024110214}.
Once the characteristic equation \(a_1(t,y,\tau,\eta)=\lambda\) has been solved in the form $\tau=a_\lambda(t,y,\eta)$,
we define the associated transverse phase-space patch by
\[
    B_\lambda
    :=
    \left\{
        (t,y,\eta):
        \bigl(t,y;a_\lambda(t,y,\eta),\eta\bigr)
        \in\mathcal B_\lambda
    \right\}
    \subset I\times T^*\mathbb R^{d-1}.
\]
For a transverse symbol \(\theta_\lambda\), set
\[
    \widehat\theta_\lambda(t,y,\zeta)
    :=
    \theta_\lambda(t,y,\lambda\zeta),
    \qquad
    \widehat B_\lambda
    :=
    \left\{
        (t,y,\zeta):
        (t,y,\lambda\zeta)\in B_\lambda
    \right\}.
\]
We say that \(\theta_\lambda\) is uniformly adapted to \(B_\lambda\) if
\(\operatorname{supp}\widehat\theta_\lambda\) is contained in a fixed
enlargement of \(\widehat B_\lambda\) and
\(\widehat\theta_\lambda\) is bounded in \(C^\infty\), uniformly in
\(\lambda\). Equivalently, in the original variables,
\begin{equation}\label{eq:transverse-cutoff-symbol-bounds}
    \left|
        \partial_t^j
        \partial_y^\alpha
        \partial_\eta^\beta
        \theta_\lambda(t,y,\eta)
    \right|
    \leq
    C_{j,\alpha,\beta}\lambda^{-|\beta|}
\end{equation}
for every \(j\geq0\) and all multi-indices \(\alpha,\beta\), in addition
to the stated support condition. For each fixed \(t\), we write
\[
    \theta_{\lambda,t}^w
    :=
    \operatorname{Op}_y^w
    \bigl(\theta_\lambda(t,y,\eta)\bigr).
\]

For two uniformly adapted families, we write
\[
    \chi_\lambda\prec\widetilde\chi_\lambda
\]
if there exists \(c>0\), independent of \(\lambda\), such that
\[
    \operatorname{dist}_{t,y,\zeta}\!\left(
        \operatorname{supp}\widehat\chi_\lambda,
        \operatorname{supp}
        \bigl(1-\widehat{\widetilde\chi}_\lambda\bigr)
    \right)
    \geq c.
\]
Thus \(\widetilde\chi_\lambda=1\) on a uniform neighborhood of
\(\operatorname{supp}\chi_\lambda\) in the rescaled phase-space variables.
%
Finally, for a measure space \(X\), define
\begin{equation}\label{eq:Ylambda}
    \mathcal Y_\lambda(X)
    :=
    L^2(X)+\lambda^\gamma L^p(X),
    \qquad
    \|h\|_{\mathcal Y_\lambda(X)}
    :=
    \inf_{h=h_2+h_p}
    \left(
        \|h_2\|_{L^2(X)}
        +
        \lambda^\gamma\|h_p\|_{L^p(X)}
    \right).
\end{equation}
This notation allows the \(L^2\)-remainder and the \(L^p\)-forcing term to
be estimated simultaneously.

\begin{proposition}[Localized Christ--Kiselev upgrade]\label{prop:CK-upgrade}
Let \(1<p\leq2\leq p'<\infty\), with \(p\) and \(p'\) conjugate, let
\(\gamma\geq0\), and let \(\lambda\geq1\). Let
\(I=(a,b)\subset\mathbb R\) be an interval of length at most one, set
\(D_t=-i\partial_t\), and let

\[
    \mathcal L_\lambda:=D_t-A_\lambda(t),
    \qquad
    A_\lambda(t)
    :=
    \operatorname{Op}_y^w\bigl(a_\lambda(t,y,\eta)\bigr),
\]

where \(a_\lambda\) is real-valued, \(A_\lambda(t)\) has a unitary
two-parameter propagator \(S(t,s)\) on
\(L^2(\mathbb R^{d-1})\), and
\begin{equation}\label{eq:A-lambda-L2-bound}
    \sup_{t\in I}
    \|A_\lambda(t)\|_{L^2_y\to L^2_y}
    \lesssim \lambda.
\end{equation}
Fix two real-valued families of
uniformly \(B_\lambda\)-adapted symbols
\[
    \chi_\lambda\prec\widetilde\chi_\lambda.
\]
Assume that there is a \(\delta>0\), independent of \(\lambda\), such that
the time projection of \(\operatorname{supp}\chi_\lambda\) is contained in
\([a+\delta,b-\delta]\).
Suppose that the homogeneous estimate

\begin{equation}\label{eq:homogeneous-first-order}
    \|\theta_{\lambda,t}^w \tilde{v}\|_{L^{p'}_{t,y}(I\times\mathbb R^{d-1})}
    \lesssim
    \lambda^\gamma
    \bigl(
        \|\tilde{v}\|_{L^2_{t,y}(I\times\mathbb R^{d-1})}
        +\|\mathcal L_\lambda \tilde{v}\|_{L^2_{t,y}(I\times\mathbb R^{d-1})}
    \bigr).
\end{equation}

holds, for every \(\tilde{v}\) in the graph domain of \(\mathcal L_\lambda\) in
\(L^2(I\times\mathbb R^{d-1})\), for the two fixed choices
$\theta_\lambda=\chi_\lambda$ and $\theta_\lambda=\widetilde\chi_\lambda$. Assume in addition that
\begin{equation}\label{eq:cutoff-commutator-L2}
    \sup_{t\in I}
    \bigl\|
        [\mathcal L_\lambda,
        \widetilde\chi_{\lambda,t}^w]
    \bigr\|_{L^2_y\to L^2_y}
    \lesssim 1.
\end{equation}
Then, every \(v\in L^2(I\times\mathbb R^{d-1})\) such that
\(\mathcal L_\lambda v\in
\mathcal Y_\lambda(I\times\mathbb R^{d-1})\), satisfies
\begin{equation}\label{eq:CK-direct}
    \|\chi_{\lambda,t}^w v\|_{L^{p'}_{t,y}(I\times\mathbb R^{d-1})}
    \lesssim
    \lambda^\gamma
    \left(
        \|v\|_{L^2_{t,y}(I\times\mathbb R^{d-1})}
        +\|\mathcal L_\lambda v\|
            _{\mathcal Y_\lambda(I\times\mathbb R^{d-1})}
    \right)
    .
\end{equation}
\end{proposition}

\begin{proof}
\textbf{Step 1: homogeneous and inhomogeneous estimates.}
Let \(S(t,s)\) be the unitary propagator of \(A_\lambda(t)\). 
For
\(\theta_\lambda\in\{\chi_\lambda,\widetilde\chi_\lambda\}\), apply
\eqref{eq:homogeneous-first-order} to \(z(t)=S(t,s)h\). Since
\(\mathcal L_\lambda z=0\) and \(|I|\leq1\), this gives, uniformly in
\(s\in I\),
\begin{equation}\label{eq:homogeneous-flow}
    \|\theta_{\lambda,t}^wS(t,s)h\|
        _{L^{p'}_{t,y}(I\times\mathbb R^{d-1})}
    \lesssim
    \lambda^\gamma\|h\|_{L^2_y(\mathbb R^{d-1})}.
\end{equation}
Then, a \(TT^*\)-argument and equation \eqref{eq:homogeneous-flow} yields

\begin{equation}\label{eq:full-Duhamel}
    \left\|
        \int_I
        \chi_{\lambda,t}^wS(t,s)
        \widetilde\chi_{\lambda,s}^w f(s)\,ds
    \right\|_{L^{p'}_{t,y}(I\times\mathbb R^{d-1})}
    \lesssim
    \lambda^{2\gamma}
    \|f\|_{L^p_{t,y}(I\times\mathbb R^{d-1})}.
\end{equation}

If \(p'>2\), then \(p<p'\), and the Christ--Kiselev Lemma, see e.g. 
\cite{CHRIST2001409,Tao01011999} converts the full Duhamel operator into the
retarded Duhamel operator without changing the magnitude of the estimate:

\begin{equation}\label{eq:retarded-Duhamel}
    \left\|
        \int_{I\cap\{s<t\}}
        \chi_{\lambda,t}^wS(t,s)
        \widetilde\chi_{\lambda,s}^w f(s)\,ds
    \right\|_{L^{p'}_{t,y}(I\times\mathbb R^{d-1})}
    \lesssim
    \lambda^{2\gamma}
    \|f\|_{L^p_{t,y}(I\times\mathbb R^{d-1})}.
\end{equation}

\textbf{Step 2: Christ-Kiselev type estimate.}
For \(p'=2\), one has \(p=2\). In this case the retarded estimate follows
directly from the uniform \(L^2\)-boundedness of the two cutoffs, unitarity,
Cauchy--Schwarz in \(s\), and \(|I|\leq1\). This bound is stronger than the
one displayed in \eqref{eq:retarded-Duhamel}, since \(\lambda\geq1\) and
\(\gamma\geq0\), and no endpoint form of the Christ--Kiselev Lemma is needed.
Moreover, Minkowski's
inequality and \eqref{eq:homogeneous-flow} give

\begin{equation}\label{eq:L2-Duhamel}
    \left\|
        \int_{I\cap\{s<t\}}
        \chi_{\lambda,t}^wS(t,s)g(s)\,ds
    \right\|_{L^{p'}_{t,y}(I\times\mathbb R^{d-1})}
    \lesssim
    \lambda^\gamma
    \|g\|_{L^2_{t,y}(I\times\mathbb R^{d-1})}.
\end{equation}

\textbf{Step 3: Frequency scale splitting:}
Choose \(t_0\in I\) before the time projection of
\(\operatorname{supp}\chi_\lambda\), and choose a real-valued cutoff
\(\eta\in C_c^\infty(I)\) such that $\eta(t_0)=0,$ $\eta=1$ on the time projection of $\operatorname{supp}\chi_\lambda$.
By the time-support assumption in the proposition, \(t_0\) and \(\eta\) can
be chosen with
\(\|\eta'\|_\infty\lesssim1\), uniformly in \(\lambda\). Since
\(\chi_{\lambda,t}^w u=\chi_{\lambda,t}^w(\eta u)\), we apply Duhamel's
formula to \(\eta u\), whose value at \(t_0\) is zero. 
We split the following:
we have:

\begin{align}\label{eq: charactersitic splitting}
\begin{split}
    \|\chi_{\lambda,t}^w u\|_{L^{p'}_{t,y}(I\times\mathbb R^{d-1})}
&\le\|\chi_{\lambda,t}^w\widetilde{\chi}_{\lambda,t}^w u\|_{L^{p'}_{t,y}(I\times\mathbb R^{d-1})}+\|\chi_{\lambda,t}^w(I-\widetilde{\chi}_{\lambda,t}^w) u\|_{L^{p'}_{t,y}(I\times\mathbb R^{d-1})}\\
&=:I+II
\end{split}    
\end{align}

\textbf{Step 3.1} We estimate $I$ which in Step 3.2 will turn out to be only contributing term.

Let
\[
    \mathcal L_\lambda u=h_2+h_p,
    \qquad
    h_2\in L^2_{t,y},
    \qquad
    h_p\in L^p_{t,y},
\]
be any admissible decomposition.

Since
\[
    \mathcal L_\lambda\left(\widetilde\chi_{\lambda,t}^w(\eta u)\right)
    =
    \eta\widetilde\chi_{\lambda,t}^wh_p
    +\eta\widetilde\chi_{\lambda,t}^wh_2
    +\eta[\mathcal L_\lambda,\widetilde\chi_{\lambda,t}^w]u+\chi_{\lambda,t}^w[D_t,\eta]u,\quad [D_t,\eta]=-i\eta'
\]
the uniform \(L^2\)-boundedness of
\(\widetilde\chi_{\lambda,t}^w\),
Duhamel's formula, \eqref{eq:cutoff-commutator-L2}, \eqref{eq:retarded-Duhamel}, and
\eqref{eq:L2-Duhamel} now give the following estimate.

\begin{align}
\begin{split}
I&\le
\left\|
    \int_{t_0}^{t}
    \chi_{\lambda,t}^wS(t,s)\widetilde\chi_{\lambda,s}^w
    \bigl(\eta h_p\bigr)(s)\,ds
\right\|_{L^{p'}_{t,y}(I\times\mathbb R^{d-1})}
\\
&\quad+
\left\|
    \int_{t_0}^{t}
    \chi_{\lambda,t}^wS(t,s)\eta(s)\widetilde\chi_{\lambda,s}^wh_2(s)\,ds
\right\|_{L^{p'}_{t,y}(I\times\mathbb R^{d-1})}
\\
&\quad+
\left\|
    \int_{t_0}^{t}
    \chi_{\lambda,t}^wS(t,s)\eta(s)
    [\mathcal L_\lambda,\widetilde\chi_{\lambda,t}^w]u
    (s)\,ds
\right\|_{L^{p'}_{t,y}(I\times\mathbb R^{d-1})}
\\
&\quad+
\left\|
    \int_{t_0}^{t}
    \chi_{\lambda,t}^wS(t,s)[D_t,\eta]u\,ds
\right\|_{L^{p'}_{t,y}(I\times\mathbb R^{d-1})}
\\
&\lesssim
\lambda^\gamma\|h_2\|_{L^2_{t,y}(I\times\mathbb R^{d-1})}
+
\lambda^{2\gamma}\|h_p\|_{L^p_{t,y}(I\times\mathbb R^{d-1})}
+
\lambda^\gamma\|u\|_{L^2_{t,y}(I\times\mathbb R^{d-1})}.
\end{split}
\end{align}

\textbf{Step 3.2:} We show that the remainder \(II\) is \(O(\lambda^{-\infty})\), which
will conclude the proof. Put
\begin{equation}\label{eq: smoothingpartsthat can be ignored}
  R_{\lambda,t}
    :=
    \chi_{\lambda,t}^w
    (I-\widetilde\chi_{\lambda,t}^w).  
\end{equation}
By the separated-support symbolic calculus and the strengthened uniform
nesting, for every \(K,N\geq0\),
\begin{equation}\label{eq:zworski-smoothing}
    \sup_{t\in I}
    \left(
        \|R_{\lambda,t}\|_{L^2_y\to H^K_y}
        +\|R_{\lambda,t}\|_{L^p_y\to H^K_y}
        +\|\partial_tR_{\lambda,t}\|_{L^2_y\to H^K_y}
        +\|\partial_tR_{\lambda,t}\|_{L^p_y\to H^K_y}
    \right)
    \leq C_{K,N}\lambda^{-N}.
\end{equation}
Here one may use the separated-support estimate
\cite[Theorem~4.25]{zw22} together with the residual calculus
\cite[\S8.4, especially (8.4.18)]{zw22}; the same argument applies after
one \(t\)-derivative because the symbol families and their nesting are
uniform in \(t\).

Choose \(K\) sufficiently large that
\begin{equation}\label{eq:Lp-Hk embedding}
    H_y^K(\mathbb R^{d-1})
    \hookrightarrow
    L_y^{p'}(\mathbb R^{d-1}).
\end{equation}
Then
\begin{align*}
II&=\|R_{\lambda,t}u\|_{L^{p'}(I\times Y)}
\\
&\lesssim \sup_{t\in I}\|R_{\lambda,t}u(t)\|_{L_y^{p'}}
\\
&\lesssim
\sup_{t\in I}
\|R_{\lambda,t}u(t)\|_{H_y^K}
\\
&\lesssim
\|R_{\lambda,\cdot}u\|_{W^{1,p}(I;H_y^K)}
\\
&\lesssim_{K,N,p'}\lambda^{-N}
\left(
\|u\|_{L^2}
+
\|h_2\|_{L^2}
+
\|h_p\|_{L^p}
\right).
\end{align*}
Line three follows from from the embedding \eqref{eq:Lp-Hk embedding}, line four is the one-dimensional embedding
$W^{1,p}(I;H^K_y)\hookrightarrow L^\infty(I;H^K_y)$ for $p>1$, and line five follows from $D_tv=A_\lambda(t)v+h_2+h_p$, the assumption \eqref{eq:A-lambda-L2-bound}, and \eqref{eq:zworski-smoothing}.  
\end{proof}

\subsection{What the spectral-cluster estimate supplies}

Set
\[
    A:=P^{1/m},
\]
where the positive \(m\)-th root is defined by the spectral calculus. Thus
\(A\) is a positive self-adjoint elliptic pseudodifferential operator of order
one.

We record only the two consequences of \eqref{Sogge1} that are needed below.

\begin{lemma}\label{lem:cluster-consequences}
Let \(\lambda\geq2\). First, if
\(v=\Pi_{[0,2\lambda]}v\), then

\begin{equation}\label{eq:quasimode-A}
    \|v\|_{L^{p'}(M)}
    \lesssim
    \lambda^\gamma
    \bigl(
        \|v\|_{L^2(M)}
        +\|(A-\lambda)v\|_{L^2(M)}
    \bigr).
\end{equation}

Second, with

\[
    R_\lambda^\pm
    :=\bigl(P-(\lambda\pm i)^m\bigr)^{-1},
\]

one has the two half-resolvent bounds

\begin{equation}\label{eq:half-resolvent}
    \|R_\lambda^\pm\|_{L^2(M)\to L^{p'}(M)}
    +
    \|R_\lambda^\pm\|_{L^p(M)\to L^2(M)}
    \lesssim
    \lambda^{\gamma+1-m}.
\end{equation}
\end{lemma}

\begin{proof}
By \eqref{Sogge1}, the triangle inequality,
Cauchy--Schwarz with the weights \(1+|k-\lambda|\), and the spectral theorem, we have

\begin{align*}
\begin{split}
    \|v\|_{L^{p'}(M)}
    &\leq
    \sum_{0\leq k\leq 2\lambda}
        \|\Pi_{k}v\|_{L^{p'}(M)}
    \\
    &\lesssim
    \lambda^\gamma
    \sum_{0\leq k\leq 2\lambda}
        \|\Pi_{k}v\|_{L^2(M)}
    \\
    &\lesssim
    \lambda^\gamma
    \left(
        \sum_{k\geq0}
        (1+|k-\lambda|)^2
        \|\Pi_{k}v\|_{L^2(M)}^2
    \right)^{1/2}
    \\
    &\lesssim
    \lambda^\gamma\left(\bigl\|(1+|A-\lambda|)v\bigr\|_{L^2(M)}^2\right)^{1/2}
    \\
    &\lesssim
    \lambda^\gamma
    \left(
        \|v\|_{L^2(M)}
        +
        \|(A-\lambda)v\|_{L^2(M)}
    \right),      
\end{split}
\end{align*}

For the first estimate in \eqref{eq:half-resolvent}, the same decomposition
gives

\[
    \|R_\lambda^\pm f\|_{L^{p'}(M)}
    \lesssim
    \left(
        \sum_{k\geq0}
        \langle k\rangle^{2\gamma}
        \sup_{\tau\in I_k}
        \frac{1}{|\tau^m-(\lambda\pm i)^m|^2}
    \right)^{1/2}
    \|f\|_{L^2(M)}.
\]

For \(\tau\in[\lambda/2,2\lambda]\),

\[
    |\tau^m-(\lambda\pm i)^m|
    \gtrsim
    \lambda^{m-1}(1+|\tau-\lambda|),
\]

whereas outside this interval the denominator is bounded below by a constant
multiple of \(\lambda^m+\tau^m\). Consequently, the displayed square sum is
\(O(\lambda^{2\gamma+2-2m})\). The high-frequency tail is summable because
the admissible range implies

\[
    \gamma<\frac{m-1}{2}.
\]

This proves the \(L^2(M)\to L^{p'}(M)\) estimate. The
\(L^p(M)\to L^2(M)\) estimate follows
by duality, since \((R_\lambda^\pm)^*=R_\lambda^\mp\). This is also the
unit-window case of the abstract equivalences in Cuenin
\cite[Section 3]{CUENIN2024110214}.
\end{proof}

\subsection{The characteristic part}
We divide the proof of the characteristic estimate into four steps:
spectral and coordinate localization, reduction to a first-order evolution,
application of Proposition \ref{prop:CK-upgrade}, and passage back to the
original \(m\)-th order operator.

\subsubsection{Spectral and coordinate localization}
Choose \(\psi\in C_c^\infty((1/2,2))\), supported in a sufficiently small
fixed neighborhood of \(1\), such that \(\psi=1\) on a smaller neighborhood
of \(1\), and set
\[
    u_{\mathrm{ch}}:=\psi(A/\lambda)u.
\]
The $\lambda$-dependent wave-front set of
\(u_{\mathrm{ch}}\) is contained in a fixed small neighborhood of the
characteristic set. The support of \(\psi\) is chosen small enough that this
neighborhood is covered by the characteristic patches constructed below.

We cover \(M\) by finitely many coordinate patches \(U_j\), with coordinate
maps
\[
    \kappa_j:U_j\longrightarrow V_j\subset\mathbb R^d,
\]
and choose a subordinate partition of unity \((\rho_j)_j\). If
\(dV_g=\omega_j(x)\,dx\) in the \(j\)-th patch, let
\[
    T_jf(x):=\omega_j(x)^{1/2}f\bigl(\kappa_j^{-1}(x)\bigr).
\]
For every \(1<r<\infty\),
\[
    \|T_jf\|_{L^r(V_j)}\asymp\|f\|_{L^r(U_j)},
\]
with constants uniform over the finite atlas. We choose the coordinate
supports inside sets of the form \(I\times\mathbb R^{d-1}\), where
\(|I|\leq1\), at a fixed positive distance from the endpoints of \(I\), and
extend localized functions by zero. The same comparison holds for the sum
space defined in \eqref{eq:Ylambda}:
\begin{equation}\label{eq:local-Y-comparison}
    \|T_j(\rho_jF)\|_{\mathcal Y_\lambda(I\times\mathbb R^{d-1})}
    \lesssim
    \|F\|_{\mathcal Y_\lambda(M)}.
\end{equation}
Indeed, \(T_j\rho_j\) is bounded separately on \(L^2\) and \(L^p\), so
\eqref{eq:local-Y-comparison} follows by applying these two bounds to an
arbitrary admissible decomposition of \(F\) and taking the infimum.

\subsubsection{First-order reduction and verification of the localized hypotheses}
Let \(A_{\mathrm{loc}}\) be the density-conjugated local representative of
\(A\). Its principal symbol is
\[
    a_1(x,\xi)=p_m(x,\xi)^{1/m}.
\]
Fix a point \((x_0,\xi_0)\) on \(\{a_1=\lambda\}\). Since \(a_1\) is
positive and homogeneous of degree one in \(\xi\), Euler's identity implies
\(\partial_\xi a_1(x_0,\xi_0)\neq0\); see, for example,
\cite[Proposition 0.5.4]{CS}. After a linear change of coordinates, write
\[
    x=(t,y),
    \qquad
    \xi=(\tau,\eta),
\]
so that \(\partial_\tau a_1\neq0\) at the chosen point. Shrinking the
coordinate and conic neighborhoods, we may assume that
\[
    |\partial_\tau a_1(t,y,\tau,\eta)|\geq c>0
\]
throughout the patch contained in the region
\(\mathcal B_\lambda\). The implicit function theorem then gives a real-valued
smooth function \(a_\lambda(t,y,\eta)\) such that
\[
    a_1\bigl(t,y,a_\lambda(t,y,\eta),\eta\bigr)=\lambda.
\]
Moreover,
\begin{equation}\label{eq:first-order-symbol-factorization}
    a_1(t,y,\tau,\eta)-\lambda
    =e_\lambda(t,y,\tau,\eta)
      \bigl(\tau-a_\lambda(t,y,\eta)\bigr),
\end{equation}
where
\[
    e_\lambda(t,y,\tau,\eta)
    :=\int_0^1
        \partial_\tau a_1\bigl(
            t,y,
            a_\lambda(t,y,\eta)
            +s(\tau-a_\lambda(t,y,\eta)),
            \eta
        \bigr)\,ds.
\]
After one further shrinking, \(e_\lambda\) is bounded away from zero and is
therefore elliptic of order zero on the patch.

We extend \(a_\lambda\) from the projection of the smaller patch to a
globally defined real-valued symbol which agrees with the local root near that
projection, equals \(\lambda\) for \(|\eta|\geq C\lambda\), and is constant
outside a fixed compact set in \((t,y)\). The extension can be chosen so that
\begin{equation}\label{eq:a-lambda-symbol-bounds}
    \bigl|
        \partial_t^j\partial_y^\alpha\partial_\eta^\beta
        a_\lambda(t,y,\eta)
    \bigr|
    \lesssim_{j,\alpha,\beta}
    \lambda^{1-|\beta|}.
\end{equation}
Define
\[
    A_\lambda(t):=
    \operatorname{Op}_y^w\bigl(a_\lambda(t,y,\eta)\bigr),
    \qquad
    \mathcal L_\lambda:=D_t-A_\lambda(t).
\]
Writing \(a_\lambda=\lambda+b_\lambda\), where \(b_\lambda\) is supported
in \(|\eta|\lesssim\lambda\), and applying the Calder\'on--Vaillancourt Theorem, see for example
\cite[Theorem 4.23]{zw22}, to \eqref{eq:a-lambda-symbol-bounds}, we
obtain
\begin{equation}\label{eq:verified-A-lambda-L2-bound}
    \sup_{t\in I}
    \|A_\lambda(t)\|_{L^2_y\to L^2_y}
    \lesssim\lambda.
\end{equation}
The symbol is real and the quantization is Weyl, so \(A_\lambda(t)\) is
bounded and self-adjoint. The estimates with one \(t\)-derivative imply norm
continuity of \(t\mapsto A_\lambda(t)\). Hence by Dyson expansion \(A_\lambda(t)\) generates a
unitary two-parameter propagator \(S(t,s)\); see
\cite[Chapter X.12]{RSII}. This verifies the propagator and operator-bound
hypotheses of Proposition \ref{prop:CK-upgrade}.

For this patch, fix two real-valued transverse cutoffs
\[
    \chi_\lambda\prec\widetilde\chi_\lambda
\]
whose lifted supports are compactly contained in the region where
\eqref{eq:first-order-symbol-factorization} holds. We choose their time
projections inside a fixed compact subinterval of \(I\). Constructing them
from fixed compactly supported functions after the rescaling
\(\eta=\lambda\xi\) gives the uniform symbol bounds
\eqref{eq:transverse-cutoff-symbol-bounds} and the uniform
nesting at frequency scale $\lambda$ required in Proposition \ref{prop:CK-upgrade}.

We next verify its commutator hypothesis. Set
\[
    \alpha_\lambda(t,y,\xi)
    :=\lambda^{-1}a_\lambda(t,y,\lambda\xi),
    \qquad
    q_\lambda(t,y,\xi)
    :=\widetilde\chi_\lambda(t,y,\lambda\xi).
\]
By \eqref{eq:a-lambda-symbol-bounds} and
\eqref{eq:transverse-cutoff-symbol-bounds}, both families are bounded in
\(S^0\), and
\[
    A_\lambda(t)=\lambda\operatorname{Op}_h^w(\alpha_\lambda(t)),
    \qquad
    \widetilde\chi_{\lambda,t}^w
    =\operatorname{Op}_h^w(q_\lambda(t)).
\]
Weyl calculus therefore gives
\[
    [A_\lambda(t),\widetilde\chi_{\lambda,t}^w]
    \in\Psi_h^0
\]
uniformly: the commutator of the two order-zero operators gains
a factor \( \lambda^{-1}\), which cancels the prefactor \(\lambda\). Also,
because \(D_t=-i\partial_t\),
\[
    [D_t,\widetilde\chi_{\lambda,t}^w]
    =-i(\partial_t\widetilde\chi_\lambda)_t^w
    \in\Psi_h^0
\]
uniformly. Applying Calder\'on--Vaillancourt again, now yields
\begin{equation}\label{eq:verified-cutoff-commutator}
    \sup_{t\in I}
    \bigl\|
        [\mathcal L_\lambda,
         \widetilde\chi_{\lambda,t}^w]
    \bigr\|_{L^2_y\to L^2_y}
    \lesssim1.
\end{equation}

We now record the microlocal factorization used twice below. 
If the phase-space cutoffs, 
\(\Theta_{0,\lambda}\prec\Theta_{1,\lambda}\prec
\Theta_{2,\lambda}\), quantized in all \((t,y)\)-variables, are supported in the present patch, then
\eqref{eq:first-order-symbol-factorization} and the parameter-dependent Weyl
calculus give order-zero operators \(E_\lambda,F_\lambda\) and uniformly
\(L^2\)-bounded remainders \(R_\lambda,R_\lambda'\) such that, microlocally
on the support of \(\Theta_{0,\lambda}\),
\begin{subequations}\label{eq:quantized-first-order-factorizations}
\begin{align}
    \Theta_{0,\lambda}^w(A_{\mathrm{loc}}-\lambda)
        \Theta_{1,\lambda}^w
    &=\Theta_{0,\lambda}^wE_\lambda\Theta_{2,\lambda}^w
        \mathcal L_\lambda\Theta_{1,\lambda}^w+R_\lambda,
        \label{eq:forward-first-order-factorization}
    \\
    \Theta_{0,\lambda}^w\mathcal L_\lambda
        \Theta_{1,\lambda}^w
    &=\Theta_{0,\lambda}^wF_\lambda\Theta_{2,\lambda}^w
        (A_{\mathrm{loc}}-\lambda)
        \Theta_{1,\lambda}^w+R_\lambda'.
        \label{eq:converse-first-order-factorization}
\end{align}
\end{subequations}
Here \(E_\lambda\) has principal symbol \(e_\lambda\), whereas
\(F_\lambda\) has principal symbol \(e_\lambda^{-1}\), after harmless
cutoffs have been incorporated into these symbols. Commutators with the
nested cutoffs and the subprincipal part of \(A_{\mathrm{loc}}\) have order
zero and are included in the remainders. Terms involving separated cutoffs
are \(O(\lambda^{-N})\) between arbitrary $\lambda$-dependent Sobolev spaces, for
every \(N\). This proves \eqref{eq:quantized-first-order-factorizations}; see
also \cite[Chapters 4 and 12]{zw22}.

We verify the homogeneous hypothesis
\eqref{eq:homogeneous-first-order} for the two fixed transverse cutoffs. Let
\(\theta_\lambda\) denote either \(\chi_\lambda\) or
\(\widetilde\chi_\lambda\). Choose a full cutoff
\(\Theta_\lambda(t,y,\tau,\eta)\), supported in the present patch, whose
symbol equals \(\theta_\lambda(t,y,\eta)\) in a fixed neighborhood of
\(\tau=a_\lambda(t,y,\eta)\). A spectral cutoff which equals one on the
frequency support of \(\Theta_\lambda\) may be inserted modulo an
\(O(\lambda^{-N})\) smoothing operator for every \(N\). Thus the
spectral-localization condition in \eqref{eq:quasimode-A} is satisfied.
The local form of that estimate and
\eqref{eq:forward-first-order-factorization} give
\[
\begin{aligned}
    \|\Theta_\lambda^w \tilde{v}\|_{L^{p'}_{t,y}(I\times\R^{d-1})}
    &\lesssim
    \lambda^\gamma
    \left(
        \|\Theta_\lambda^w\tilde{v}\|_{L^2_{t,y}(I\times\R^{d-1})}
        +\|(A_{\mathrm{loc}}-\lambda)
              \Theta_\lambda^w\tilde{v}\|_{L^2_{t,y}(I\times\R^{d-1})}
    \right)
    \\
    &\lesssim
    \lambda^\gamma
    \left(
        \|\tilde{v}\|_{L^2_{t,y}(I\times\R^{d-1})}
        +\|\mathcal L_\lambda \tilde{v}\|_{L^2_{t,y}(I\times\R^{d-1})}
    \right).
\end{aligned}
\]
In the second line, commutators with the nested cutoffs and all lower-order
terms are absorbed into \(\|\tilde{v}\|_2\).


The full symbol of
\(\theta_{\lambda,t}^w-\Theta_\lambda^w\) is supported where
\(\tau-a_\lambda(t,y,\eta)\) is elliptic. Hence a
parameter-dependent parametrix gives
\[
    \theta_{\lambda,t}^w-\Theta_\lambda^w
    =Q_\lambda\mathcal L_\lambda+R_\lambda,
\]
where \(Q_\lambda\) has order \(-1\) on frequencies of size \(\lambda\)
and \(R_\lambda\) is rapidly smoothing. Parameter-dependent Sobolev
embedding, see the arguments in Step 3.2 in Proposition \ref{prop:CK-upgrade}, therefore gives
\[
    \|Q_\lambda f\|_{L^{p'}_{t,y}(I\times\R^{d-1})}
    +\|R_\lambda \tilde{v}\|_{L^{p'}_{t,y}(I\times\R^{d-1})}
    \lesssim
    \lambda^\gamma
    \bigl(\|f\|_{L^2_{t,y}(I\times\R^{d-1})}+\|\tilde{v}\|_{L^2_{t,y}(I\times\R^{d-1})}\bigr);
\]
here one uses
\(d(1/2-1/p')-1\leq\gamma\), which holds in the admissible range.
Applying this with \(f=\mathcal L_\lambda \tilde{v}\) proves that the complementary
term satisfies the same estimate. Consequently,
\begin{equation}\label{eq:verified-homogeneous-first-order}
    \|\theta_{\lambda,t}^w\tilde{v}\|_{L^{p'}_{t,y}(I\times\mathbb R^{d-1})}
    \lesssim
    \lambda^\gamma
    \left(
        \|\tilde{v}\|_{L^2_{t,y}(I\times\mathbb R^{d-1})}
        +\|\mathcal L_\lambda \tilde{v}\|_{L^2_{t,y}(I\times\mathbb R^{d-1})}
    \right)
\end{equation}
for \(\theta_\lambda=\chi_\lambda\) and
\(\theta_\lambda=\widetilde\chi_\lambda\), with uniform constants. Thus \eqref{eq:homogeneous-first-order} holds, and hence all the assumptions of Proposition \ref{prop:CK-upgrade} are verified.

\subsubsection{Application of the localized Christ--Kiselev estimate}
We turn to the function to which the proposition will be applied. Fix one
coordinate patch and one of the finitely many microlocal patches, suppress
their indices, and write
\[
    v_{\mathrm{loc}}:=T_j(\rho_j u_{\mathrm{ch}}).
\]
Since \([A,\rho_j]\) has order zero,
\begin{equation}\label{eq:local-A-residual}
    (A_{\mathrm{loc}}-\lambda)v_{\mathrm{loc}}
    =
    T_j\bigl(\rho_j(A-\lambda)u_{\mathrm{ch}}\bigr)+r_j,
    \qquad
    \|r_j\|_{L^2_{t,y}(I\times\R^{d-1})}
    \lesssim
    \|u\|_{L^2(M)}.
\end{equation}
The function
\[
    \widetilde\chi_{\lambda,t}^wv_{\mathrm{loc}}
\]
has time support compactly contained in \(I\). After inserting the full
phase-space cutoffs around the lifted support of
\(\widetilde\chi_\lambda\), the converse factorization
\eqref{eq:converse-first-order-factorization} and
\eqref{eq:local-A-residual} give
\[
\begin{aligned}
    \widetilde\chi_{\lambda,t}^w
        \mathcal L_\lambda v_{\mathrm{loc}}
    &=
    B_\lambda
    (A_{\mathrm{loc}}-\lambda)v_{\mathrm{loc}}
    +
    G_\lambda v_{\mathrm{loc}}
    \\
    &=
    B_\lambda T_j\bigl(\rho_j(A-\lambda)u_{\mathrm{ch}}\bigr)
    +r_\lambda,
\end{aligned}
\]
where
\[
    r_\lambda:=B_\lambda r_j+G_\lambda v_{\mathrm{loc}},
\]
and \(B_\lambda\) and \(G_\lambda\) are properly supported, uniformly
parameter-dependent operators of order zero; rapidly decaying smoothing
terms are absorbed into \(G_\lambda\). Consequently,
\[
    \|r_\lambda\|_{\mathcal Y_\lambda(I\times\R^{d-1})}
    \lesssim
    \|u\|_{L^2(M)}.
\]
We now split
\begin{align}\label{eq:characteristic-splitting}
\begin{split}
    \|\chi_{\lambda,t}^wv_{\mathrm{loc}}\|
        _{L^{p'}_{t,y}(I\times\R^{d-1})}
    &\leq
    \|\chi_{\lambda,t}^w
        \widetilde\chi_{\lambda,t}^wv_{\mathrm{loc}}
    \|_{L^{p'}_{t,y}(I\times\R^{d-1})}
    \\
    &\quad+
    \|\chi_{\lambda,t}^w
        (I-\widetilde\chi_{\lambda,t}^w)v_{\mathrm{loc}}
    \|_{L^{p'}_{t,y}(I\times\R^{d-1})}
    \\
    &=:I+II.
\end{split}
\end{align}
For the first term, Proposition \ref{prop:CK-upgrade} and the identity
\[
\begin{aligned}
    \mathcal L_\lambda
        \widetilde\chi_{\lambda,t}^wv_{\mathrm{loc}}
    &=
    \widetilde\chi_{\lambda,t}^w
        \mathcal L_\lambda v_{\mathrm{loc}}
    +
    [\mathcal L_\lambda,\widetilde\chi_{\lambda,t}^w]
        v_{\mathrm{loc}}
\end{aligned}
\]
yield
\begin{align*}
    I
    &\lesssim
    \lambda^\gamma
    \left(
        \|\widetilde\chi_{\lambda,t}^wv_{\mathrm{loc}}\|
            _{L^2_{t,y}(I\times\R^{d-1})}
        +
        \|\mathcal L_\lambda
            \widetilde\chi_{\lambda,t}^wv_{\mathrm{loc}}\|
            _{\mathcal Y_\lambda(I\times\R^{d-1})}
    \right)
    \\
    &\lesssim
    \lambda^\gamma
    \left(
        \|\widetilde\chi_{\lambda,t}^wv_{\mathrm{loc}}\|
            _{L^2_{t,y}(I\times\R^{d-1})}
        +
        \|\widetilde\chi_{\lambda,t}^w
            \mathcal L_\lambda v_{\mathrm{loc}}\|
            _{\mathcal Y_\lambda(I\times\R^{d-1})}
    \right.
    \\
    &\hspace{5.5cm}\left.
        +
        \|[\mathcal L_\lambda,\widetilde\chi_{\lambda,t}^w]
            v_{\mathrm{loc}}\|
            _{\mathcal Y_\lambda(I\times\R^{d-1})}
    \right)
    \\
    &\lesssim
    \lambda^\gamma
    \left(
        \|v_{\mathrm{loc}}\|
            _{L^2_{t,y}(I\times\R^{d-1})}
        +
        \|B_\lambda
            T_j\bigl(\rho_j(A-\lambda)u_{\mathrm{ch}}\bigr)\|
            _{\mathcal Y_\lambda(I\times\R^{d-1})}
    \right.
    \\
    &\hspace{3.2cm}\left.
        +
        \|r_\lambda\|
            _{\mathcal Y_\lambda(I\times\R^{d-1})}
        +
        \|[\mathcal L_\lambda,\widetilde\chi_{\lambda,t}^w]
            v_{\mathrm{loc}}\|
            _{\mathcal Y_\lambda(I\times\R^{d-1})}
    \right)
    \\
    &\lesssim
    \lambda^\gamma
    \left(
        \|u\|_{L^2(M)}
        +
        \|(A-\lambda)u_{\mathrm{ch}}\|
            _{\mathcal Y_\lambda(M)}
    \right).
\end{align*}
Here the last line follows from
\eqref{eq:local-Y-comparison},
\eqref{eq:verified-cutoff-commutator}, and the uniform boundedness of
order-zero operators on \(\mathcal Y_\lambda\).

For the second term, set
\[
    R_{\lambda,t}
    :=
    \chi_{\lambda,t}^w
    (I-\widetilde\chi_{\lambda,t}^w).
\]
Uniform nesting and the separated-support argument from Step \(3.2\) in
Proposition \ref{prop:CK-upgrade} give, for every \(N\geq0\),
\[
\begin{aligned}
    II
    &=
    \|R_{\lambda,t}v_{\mathrm{loc}}\|
        _{L^{p'}_{t,y}(I\times\R^{d-1})}
    \\
    &\leq
    C_N\lambda^{-N}\|u\|_{L^2(M)}.
\end{aligned}
\]
Combining the estimates for \(I\) and \(II\), and taking \(N\) sufficiently
large, we obtain
\begin{equation}\label{eq:characteristic-A}
    \|\chi_{\lambda,t}^wT_j(\rho_ju_{\mathrm{ch}})\|
        _{L^{p'}_{t,y}(I\times\R^{d-1})}
    \lesssim
    \lambda^\gamma
    \left(
        \|u\|_{L^2(M)}
        +
        \|(A-\lambda)u_{\mathrm{ch}}\|
            _{\mathcal Y_\lambda(M)}
    \right).
\end{equation}

\subsubsection{Passage back to the \(m\)-th order operator}

We finally rewrite the residual in terms of \(P-\lambda^m\). Define
\[
    b(r):=
    \begin{cases}
        \displaystyle
        \psi(r)\frac{r-1}{r^m-1},&r\neq1,\\[1.2ex]
        \displaystyle\frac{\psi(1)}{m},&r=1.
    \end{cases}
\]
The singularity at \(r=1\) is removable, so
\(b\in C_c^\infty((1/2,2))\). Since \(P=A^m\), the spectral theorem gives
the exact identity
\begin{equation}\label{eq:power-rewrite}
    (A-\lambda)u_{\mathrm{ch}}
    =\lambda^{1-m}b(A/\lambda)(P-\lambda^m)u.
\end{equation}
The parameter-dependent functional calculus makes \(b(A/\lambda)\) a
uniform family of order-zero pseudodifferential operators. It is therefore
uniformly bounded on both \(L^2(M)\) and \(L^p(M)\), and hence on
\(\mathcal Y_\lambda(M)\). Thus
\begin{equation}\label{eq:power-rewrite-Y}
    \|(A-\lambda)u_{\mathrm{ch}}\|_{\mathcal Y_\lambda(M)}
    \lesssim
    \lambda^{1-m}
    \|(P-\lambda^m)u\|_{\mathcal Y_\lambda(M)}.
\end{equation}

By the choice of the support of \(\psi\), finitely many of the preceding
microlocal patches cover the wave-front set of
\(u_{\mathrm{ch}}\). A microlocal partition of unity, the local norm
equivalences, and the rapidly smoothing remainder therefore give
\[
    \|u_{\mathrm{ch}}\|_{L^{p'}(M)}
    \lesssim
    \sum_{j,\nu}
    \bigl\|
        \chi_{j,\nu,\lambda}^w
        T_j(\rho_ju_{\mathrm{ch}})
    \bigr\|_{L^{p'}_{t,y}(I\times\R^{d-1})}
    +\|u\|_{L^2(M)}.
\]
The number of patches is independent of \(\lambda\). Summing
\eqref{eq:characteristic-A} and using \eqref{eq:power-rewrite-Y}, we obtain
\begin{equation}\label{eq:characteristic-P}
    \|u_{\mathrm{ch}}\|_{L^{p'}(M)}
    \lesssim
    \lambda^\gamma\|u\|_{L^2(M)}
    +\lambda^{\gamma+1-m}
    \|(P-\lambda^m)u\|_{\mathcal Y_\lambda(M)}.
\end{equation}

\subsection{Elliptic frequencies and the global direct estimate}

It remains to estimate
\[
    u_{\mathrm{ell}}:=(1-\psi(A/\lambda))u.
\]
Set
\[
    F:=(P-\lambda^m)u.
\]
We define \(T_\lambda\) by the spectral calculus as the multiplier with
symbol
\[
    m_\lambda(\mu)
    :=
    \frac{1-\psi(\mu/\lambda)}{\mu^m-\lambda^m},
    \qquad \mu\geq0.
\]
Since \(1-\psi\) vanishes in a neighborhood of \(1\), the quotient is smooth
across \(\mu=\lambda\), and the spectral theorem gives
\[
    u_{\mathrm{ell}}=T_\lambda F.
\]
Put
\[
    \alpha:=d\left(\frac12-\frac1{p'}\right).
\]
Sobolev embedding and its dual give
\begin{equation}\label{eq: sobolev-emb}
    H^\alpha(M)\hookrightarrow L^{p'}(M),
    \qquad
    L^p(M)\hookrightarrow H^{-\alpha}(M).
\end{equation}
Because \(1-\psi(\mu/\lambda)\) removes the region \(\mu\asymp\lambda\),
we have, for \(\beta=\alpha\) and \(\beta=2\alpha\),

\begin{equation}\label{eq:elliptic-multiplier-bound}
    \sup_{\mu\geq0}
    \frac{\langle\mu\rangle^\beta
    |1-\psi(\mu/\lambda)|}
    {|\mu^m-\lambda^m|}
    \lesssim
    \lambda^{\beta-m}.
\end{equation}

To see this, for \(\mu\lesssim\lambda\) away from the deleted characteristic
region one has \(|\mu^m-\lambda^m|\gtrsim\lambda^m\), whereas for
\(\mu\gtrsim\lambda\) one has
\(|\mu^m-\lambda^m|\gtrsim\mu^m\). The admissible range of exponents ensures
that \(2\alpha<m\), so the high-frequency expression is bounded.

By the spectral theorem, \eqref{eq:elliptic-multiplier-bound} implies
\[
    \|T_\lambda\|_{L^2(M)\to H^\alpha(M)}
    \lesssim
    \lambda^{\alpha-m}
\]
and
\[
    \|T_\lambda\|_{H^{-\alpha}(M)\to H^\alpha(M)}
    \lesssim
    \lambda^{2\alpha-m}.
\]
Combining these bounds with the two Sobolev embeddings \eqref{eq: sobolev-emb} gives
\begin{subequations}\label{eq:elliptic-two-bounds}
\begin{align}
    \|T_\lambda h_2\|_{L^{p'}(M)}
    &\lesssim
    \lambda^{\alpha-m}\|h_2\|_{L^2(M)},
    \\
    \|T_\lambda h_p\|_{L^{p'}(M)}
    &\lesssim
    \lambda^{2\alpha-m}\|h_p\|_{L^p(M)}.
\end{align}
\end{subequations}
From the definition of \(\gamma=\nu(p')\),
\begin{equation}\label{eq:alpha-gamma}
    \alpha\leq\gamma+\frac12.
\end{equation}
Hence, for \(\lambda\geq1\),
\[
    \lambda^{\alpha-m}
    \lesssim
    \lambda^{\gamma+1-m},
    \qquad
    \lambda^{2\alpha-m}
    \lesssim
    \lambda^{2\gamma+1-m}.
\]
Let \(F=h_2+h_p\) be an arbitrary decomposition. From
\eqref{eq:elliptic-two-bounds},
\[
\begin{aligned}
    \|u_{\mathrm{ell}}\|_{L^{p'}(M)}
    &\leq
    \|T_\lambda h_2\|_{L^{p'}(M)}
    +
    \|T_\lambda h_p\|_{L^{p'}(M)}
    \\
    &\lesssim
    \lambda^{\gamma+1-m}
    \left(
        \|h_2\|_{L^2(M)}
        +\lambda^\gamma\|h_p\|_{L^p(M)}
    \right).
\end{aligned}
\]
Taking the infimum over all decompositions \(F=h_2+h_p\) gives
\begin{equation}\label{eq:elliptic-Y}
    \|u_{\mathrm{ell}}\|_{L^{p'}(M)}
    \lesssim
    \lambda^{\gamma+1-m}
    \|(P-\lambda^m)u\|_{\mathcal Y_\lambda(M)}.
\end{equation}
Combining \eqref{eq:characteristic-P} and \eqref{eq:elliptic-Y} proves the
estimate that drives the resolvent argument.
\begin{proposition}[Global direct estimate]\label{prop:global-direct}
For every \(\lambda\geq\lambda_0\) and every smooth \(u\),
\begin{equation}\label{eq:global-direct}
    \|u\|_{L^{p'}(M)}
    \lesssim
    \lambda^\gamma\|u\|_{L^2(M)}
    +
    \lambda^{\gamma+1-m}
    \|(P-\lambda^m)u\|_{\mathcal Y_\lambda(M)}.
\end{equation}
\end{proposition}
All localization errors, commutators, and lower-order terms have been placed
in the \(L^2(M)\)-part of \(\mathcal Y_\lambda(M)\). 

\subsection{Completion of the boundary resolvent estimate}
\begin{proof}[Proof of Equation \ref{eqithm-boundary}]

Let \(f\in L^p(M)\cap L^2(M)\) and set
\begin{equation}\label{eq: resolvent equation}
     u:=R_\lambda^\pm f
    =\bigl(P-(\lambda\pm i)^m\bigr)^{-1}f.
\end{equation}
The second half-resolvent estimate in \eqref{eq:half-resolvent} gives
\begin{equation}\label{eq:u-L2-final}
    \|u\|_{L^2(M)}
    \lesssim
    \lambda^{\gamma+1-m}\|f\|_{L^p(M)}.
\end{equation}
Moreover, the resolvent equation \eqref{eq: resolvent equation} implies
\[
    (P-\lambda^m)u
    =
    f+\bigl((\lambda\pm i)^m-\lambda^m\bigr)u.
\]
We estimate the first term in the \(L^p(M)\)-part and the second term in the
\(L^2(M)\)-part of \(\mathcal Y_\lambda(M)\). Therefore,

\begin{align}\label{eq: last Resolventestimate2}
\begin{split}
    \|(P-\lambda^m)u\|_{\mathcal Y_\lambda(M)}
    &\leq
    \lambda^\gamma\|f\|_{L^p(M)}
    +
    |(\lambda\pm i)^m-\lambda^m|\,\|u\|_{L^2(M)}
    \\
    &\lesssim
    \lambda^\gamma\|f\|_{L^p(M)}
    +
    \lambda^{m-1}\|u\|_{L^2(M)}
    \\
    &\lesssim
    \lambda^\gamma\|f\|_{L^p(M)},
    \end{split}
\end{align}

where we used \eqref{eq:u-L2-final}. Inserting the last two bounds \eqref{eq:u-L2-final} and \eqref{eq: last Resolventestimate2} into
\eqref{eq:global-direct}, we obtain

\[
\begin{aligned}
    \|R_\lambda^\pm f\|_{L^{p'}(M)}
    &\lesssim
    \lambda^\gamma
    \lambda^{\gamma+1-m}\|f\|_{L^p(M)}
    +
    \lambda^{\gamma+1-m}
    \lambda^\gamma\|f\|_{L^p(M)}
    \\
    &\lesssim
    \lambda^{2\gamma+1-m}\|f\|_{L^p(M)}.
\end{aligned}
\]

By density, the estimate extends to every \(f\in L^p(M)\). Recalling that
\(\gamma=\nu(p')\), and enlarging the constant on the compact part of the
contour, we conclude that

\[
    \left\|
        \bigl(P-(\lambda\pm i)^m\bigr)^{-1}
    \right\|_{L^p(M)\to L^{p'}(M)}
    \lesssim
    \langle\lambda\rangle^{2\nu(p')+1-m}.
\]

This is \eqref{eqithm-boundary} and completes the proof.

\end{proof}

\section{Complex extension: Proof of Equation \ref{eqithm-region}}\label{sec: complex ext}
We extend the boundary estimate \eqref{eqithm-boundary} to the
whole region \(\Xi\). 

Recall that Theorem~\ref{eqithm-regionuation} gives the
uniform bound
\begin{equation}\label{eqi}
    \bigl\|
        \bigl(P-(\lambda\pm i)^m\bigr)^{-1}
    \bigr\|_{L^p(M)\to L^{p'}(M)}
    \lesssim
    \langle\lambda\rangle^{2\nu(p')+1-m}.
\end{equation}
Set
\[
    \beta:=1-\frac{2\nu(p')+1}{m}.
\]
Since \(\Xi_0\cap[0,\infty)=\emptyset\), we fix a branch of the logarithm on
\(\mathbb C\setminus[0,\infty)\), and define \(z^\beta\) with respect to this
branch.
Let \(f,g\in L^p(M)\) with $\|f\|_{L^p(M)}=\|g\|_{L^p(M)}=1.$
For \(z\in\Xi_0\), define
\[
    F_{f,g}(z)
    :=
    \left\langle
        \bigl(P-z\bigr)^{-1}f,g
    \right\rangle
    z^\beta .
\]
Since \(\Xi_0\cap[0,\infty)=\emptyset\), the resolvent $z\mapsto (P-z)^{-1}$
is holomorphic on \(\Xi_0\). Hence \(F_{f,g}\) is holomorphic on \(\Xi_0\). Moreover, \(F_{f,g}\) is continuous on \(\Gamma\),  which follows from equation \eqref{Gammacontinuous}, see Figure~\ref{fig}.

We record a crude growth estimate  in \[ \{z\in\mathbb C:\operatorname{dist}(z,\mathbb R_+)\geq 1/10\} \supset \Xi. \]
Since \((P+1)^{-1/2}\) is an elliptic operator of order \(-m/2\), the Sobolev embedding applies under the assumption \eqref{eq: admissible}, see e.g. \cite[Chapter 13]{taylor2010partial}. Together with duality, this gives \[ (P+1)^{-1/2}:L^p(M)\to L^2(M), \qquad (P+1)^{-1/2}:L^2(M)\to L^{p'}(M). \] Hence, using the factorization \[ (P-z)^{-1} = (P+1)^{-1/2} \bigl((P+1)^{1/2}(P-z)^{-1}(P+1)^{1/2}\bigr) (P+1)^{-1/2}, \] we obtain \begin{align} \begin{split} \|(P-z)^{-1}\|_{L^p(M)\to L^{p'}(M)} &\lesssim \bigl\|(P+1)^{1/2}(P-z)^{-1}(P+1)^{1/2}\bigr\|_{L^2(M)\to L^2(M)} \\ &= \sup_{\mu\in\operatorname{spec}(P)} \frac{\mu+1}{|\mu-z|} \\ &\leq 1+\frac{|z+1|}{\operatorname{dist}(z,\mathbb R_+)} \lesssim 1+|z|. \end{split} \end{align} 
This shows that $F_{f,g}$ has at most polynomial, and hence subexponential growth in this region. Therefore, the Phragmén--Lindelöf principle implies that the function \(F_{f,g}\) is bounded on \(\Xi=\Xi_0\cup\Gamma\), which yields estimate \eqref{eqithm-region}.

\end{proof}

\section{Bounds near singularities: Proof of Equation \ref{secondequation} }\label{sec: boundsnearsing}
\begin{proof}
We now consider \(z\notin\Xi\)
and prove the estimate \eqref{secondequation}. We follow the ideas of \cite{CJC}.

Throughout this step, whenever \(z\notin\Xi\), we write \(z^{1/m}\) for a
determination \(w\) satisfying $w^m=z,$ and $|\Im w|\leq 1$.
In the only place where a choice is relevant, namely on the boundary curve
\(\Gamma\), this convention is understood through the parametrization
\(z=(\lambda\pm i)^m\), so that \(z^{1/m}=\lambda\pm i\). In particular,
for \(z\in\Gamma\) one has $|\Im(z^{1/m})|=1.$
Only the real part \(\Re(z^{1/m})\) enters the estimates below, and this
quantity is unambiguous in the region under consideration.

\textbf{Step 1:} We prove \eqref{secondequation} in the case \(z\notin\Xi\) and
\(\Re(z^{1/m})\le \exp(1000m)\). The precise threshold \(\exp(1000m)\) is somewhat arbitrary. It is chosen only
so that, in the complementary regime
\(\Re(z^{1/m})>\exp(1000m)\), the complex-analytic estimates
\eqref{complexana}, \eqref{complex2}, and \eqref{complex3} become transparent.
 
 We shall use the Sobolev embedding
\begin{equation}\label{Sobolevemb}
     \|f\|_{L^{p'}(M)}
     \lesssim
     \left\|
        (P+1)^{\frac{d}{m}(\frac12-\frac1{p'})}f
     \right\|_{L^2(M)},
\end{equation}
see e.g. \cite[ch.13]{taylor2010partial} or \cite{hebey2000nonlinear,zw22}.

Now we follow the proof of \cite[Lemma~2.3]{CJC}, replacing the Sobolev embedding used there by the more general embedding \eqref{Sobolevemb}, to obtain
\begin{equation}
    \left\|\left(P-z\right)^{-1}\right\|_{L^{p}(M)\to L^{p'}(M)}
    \lesssim_{A,d,q} \frac{1}{d(z)},
    \qquad \Re(z)\le A,
\end{equation}
for any $A\ge 1$.
If $\Re(z^{1/m})\le \exp(1000m)$,
then the above implies inequality $\eqref{secondequation}$. 

\textbf{Step 2:}
We now consider \(z\notin\Xi\) with \(\Re(z^{1/m})> \exp(1000m)\) and prove \eqref{secondequation} in this regime.
To this end, it suffices to establish the following two estimates:

\begin{subequations}
        \begin{align}
        \label{eqithm-regionSogge}
\left\|
\boldsymbol{1}\!\left(
P^{1/m}\in
\bigl[\Re(z^\frac1m)-\tfrac12,\,
\Re(z^\frac1m)+\tfrac12\bigr]
\right)
\left(P-z\right)^{-1}
\right\|_{L^p(M)\to L^{p'}(M)}
&\lesssim d(z)^{-1}|z|^{\frac{2\nu(p')}{m}},
\\ \label{secondequationSogge}
\left\|
\boldsymbol{1}\!\left(
P^{1/m}\notin
\bigl[\Re(z^\frac1m)-\tfrac12,\,
\Re(z^\frac1m)+\tfrac12\bigr]
\right)
\left(P-z\right)^{-1}
\right\|_{L^p(M)\to L^{p'}(M)}
&\lesssim |z|^{\frac{2\nu(p')+1}{m}-1}.
\end{align} 
    \end{subequations}

\textit{Step 2.1:} We first prove the localized estimate \eqref{eqithm-regionSogge} by applying Sogge's spectral cluster bounds \eqref{Sogge1} and \eqref{Sogge2} with
\(\Upsilon=\Re(z^{1/m})\).

\begin{align}\label{Soggeequationchain}
\begin{split}
  &\bigl\|\mathbf{1}\bigl(P^{1/m}\in[\Re(z^\frac1m)-1/2,\Re(z^\frac1m)+1/2]\bigr)(P-z)^{-1}f\bigr\|_{L^{p'}(M)}
\\
\lesssim&
\Re(z^\frac1m)^{\nu(p')}
\bigl\|\mathbf{1}\bigl(P^{1/m}\in[\Re(z^\frac1m)-1/2,\Re(z^\frac1m)+1/2]\bigr)(P-z)^{-1}f\bigr\|_{L^2(M)}
\\
\lesssim&
d(z)^{-1}\Re(z^\frac1m)^{\nu(p')}
\bigl\|\mathbf{1}\bigl(P^{1/m}\in[\Re(z^\frac1m)-1/2,\Re(z^\frac1m)+1/2]\bigr)f\bigr\|_{L^2(M)}
\\
\lesssim&
d(z)^{-1}\Re(z^\frac1m)^{2\nu(p')}
\|f\|_{L^p(M)}.  
\end{split}
\end{align}
Since
$\Re(z^{1/m})\le |z|^{1/m}$,
we obtain the desired estimate \eqref{eqithm-regionSogge}.

\textit{Step 2.2:} We next prove the complementary estimate \eqref{secondequationSogge}.
We choose \(z_1\in\Gamma,\) hence $|\Im{(z_1^{1/m})}|=1$, such that
$\Re(z_1^{1/m})=\Re(z^{1/m})$,
and decompose the corresponding expression as follows:

\begin{subequations}
        \begin{align}
        \label{zeroJC}
&\bigl\|\mathbf{1}\bigl(P^{1/m}\notin [\Re(z^\frac1m)-1/2,\Re(z^\frac1m)+1/2]\bigr)(P-z)^{-1}f\bigr\|_{L^{p'}(M)}
\\ \label{eqithm-regionJC}
\le
&\bigl\|\mathbf{1}\bigl(P^{1/m}\notin [\Re(z^\frac1m)-1/2,\Re(z^\frac1m)+1/2]\bigr)(P-z_1)^{-1}f\bigr\|_{L^{p'}(M)}
\\ \label{secondequationJC}
+
&\bigl\|\mathbf{1}\bigl(P^{1/m}\notin [\Re(z^\frac1m)-1/2,\Re(z^\frac1m)+1/2]\bigr)
\bigl[(P-z)^{-1}-(P-z_1)^{-1}\bigr]f\bigr\|_{L^{p'}(M)}
\end{align} 
    \end{subequations}

We further estimate:
\begin{subequations}
    \begin{align} \label{I}
        \eqref{eqithm-regionJC} \le &\|(P-z_1)^{-1}f\|_{L^{p'}(M)}\\ \label{II}
        +
&\bigl\|\mathbf{1}\bigl(P^{1/m}\in [\Re(z^\frac1m)-1/2,\Re(z^\frac1m)+1/2]\bigr)(P-z_1)^{-1}f\bigr\|_{L^{p'}(M)}
    \end{align}
\end{subequations}

\textit{Step 2.2.1:} We estimate the term \eqref{I}. Since \(|\Im(z_1^{1/m})|=1\), the point \(z_1\) lies on the boundary curve \(\Gamma\), and hence we can bound the term \eqref{I} by the left hand side of \eqref{eqithm-region}.

\textit{Step 2.2.2:} We now bound the cluster term \eqref{II}.
For \(\tau\ge 0\), we use the real and imaginary parts of \(z_1\) to obtain
\begin{equation}\label{complexana}
|\tau^m-z_1|
\ge |\Im(z_1)|
= (\Re(z_1^{\frac{1}{m}})^2+1)^{m/2}\bigl|\sin(m\arctan(1/\Re(z_1^{\frac{1}{m}})))\bigr|
\gtrsim \Re(z_1^{\frac{1}{m}})^{m-1}
\gtrsim |z|^{1-\frac1m}.
\end{equation}
This implies 
\begin{equation}\label{compleximpl}
    \bigl\|\mathbf{1}\bigl(P^{1/m}\in [\Re(z^\frac1m)-1/2,\Re(z^\frac1m)+1/2]\bigr)(P-z_1)^{-1}f\bigr\|_{L^2(M)}
\lesssim |z|^{\frac{1}{m}-1}\,\|f\|_{L^2(M)}.
\end{equation}

Now $$\eqref{II}\le  |z|^{\frac{2\nu(p')+1}{m}-1}\|f\|_{L^{p}(M)}$$

follows again by the same line of argument as in inequality \eqref{Soggeequationchain}, where we exchange the estimate from the second to the third line by the estimate \eqref{compleximpl}.

\textit{Step 2.2.3:} 
It remains to bound the difference term \eqref{secondequationJC}.
By the resolvent identity,
$$\left(P-z\right)^{-1}-\left( P-z_1\right)^{-1}=(z-z_1)\left( P-z\right)^{-1}\left( P-z_1\right)^{-1},$$
inequality \eqref{secondequationJC} follows from the following estimate

  \begin{align} \label{subzero}
   \begin{split}
        &\bigl\|\mathbf{1}\bigl(P^{1/m}\notin [\Re(z^\frac1m)-1/2,\Re(z^\frac1m)+1/2]\bigr)
\bigl[(P-z)^{-1}\bigr]\bigr\|_{L^{2}(M)\to L^{p'}}\\ 
\times &\bigl\|
(P-z_1)^{-1}\bigr\|_{L^{p}(M)\to L^{2}}\\ 
\times &|z-z_1|\le|z|^{\frac{2\nu(p')+1}{m}-1},
   \end{split}
\end{align}  
which we prove in the following.

\textit{Step 2.2.3.1:} We first control the factor \(|z-z_1|\).
By Taylor expansion and the assumptions on \(z\) and \(z_1\), we have
\begin{equation}\label{complex2}
  |z-z_1|\lesssim \Re(z^\frac1m)^{m-1}\lesssim |z|^{1-\frac{1}{m}}.  
\end{equation}

\textit{Step 2.2.3.2:} We now prove the two remaining operator norm bounds appearing in \eqref{subzero}.
Using the assumptions on \(z\) and \(z_1\), we obtain
\begin{equation}\label{complex3}
    |\lambda_j^m-z_1|
\gtrsim
\bigl(\Re(z_1^{1/m})\bigr)^{m-1}
\Bigl(|\lambda_j-\Re(z_1^{1/m})|+1\Bigr).
\end{equation}
We can now follow the proofs of equations (2.12) and (2.13) in Cuenin \cite{CJC} line by line to get
\begin{align}\label{last4}
    \begin{split}
        &\bigl\|\mathbf{1}\bigl(P^{1/m}\notin [\Re(z^\frac1m)-1/2,\Re(z^\frac1m)+1/2]\bigr)
\bigl[(P-z)^{-1}\bigr]\bigr\|_{L^{2}(M)\to L^{p'}}\le |z|^{\frac{\nu(p')+1}{m}-1},\\
&\bigl\|
(P-z_1)^{-1}\bigr\|_{L^{p}(M)\to L^{2}}\le |z|^{\frac{\nu(p')+1}{m}-1},
    \end{split}
\end{align}
where we again use Sogge's spectral cluster bound \eqref{Sogge2}.
Combining \eqref{complex2} and \eqref{last4} yields \eqref{subzero}, and hence \eqref{secondequation}.
\end{proof}

\section{Spectral bounds. Proof of Theorem \ref{psithm1}}\label{sec7}
We follow the Birman--Schwinger argument used by Cuenin \cite{CJC} and adapt it
to the present pseudodifferential setting.

\begin{proof}
    
    As in \cite{CJC} Theorem \ref{psithm1} follows from the Birman-Schwinger principle, and Theorem \ref{eqithm-regionuation}.


The Birman--Schwinger principle states that \(z\) is an eigenvalue of
\(P+V\) if and only if \(-1\) is an eigenvalue of the compact
Birman--Schwinger operator
\[
K(z)=|V|^{1/2}(P-z)^{-1}\operatorname{sgn}(V)|V|^{1/2}.
\]
It follows that \(\|K(z)\|\ge 1\), and therefore
\begin{equation}
    1\le\|K(z)\|\le\|V\|_{L^q(M)}\|( P-z)^{-1}\|_{L^p(M)\to L^{p'}(M)},
\end{equation}
where we recall that $1/q=1/p-1/p'$.
Equations \eqref{eqithm-region} and \eqref{secondequation} yield
\begin{subequations}
\begin{align}\label{firstBS}
    |z|^{1-\frac1m} (1+|z|)^{-\frac{2\sigma(q)}{m}} &\lesssim \,\|V\|_{L^q(M)},\quad &\text{for}\quad z\in\Xi, \\
    \label{secondBS}
    d(z)&\lesssim (1+|z|)^{\frac{2\nu(p')}{m}}\|V\|_{L^q(M)},\quad &\text{for}\quad z\notin\Xi,
\end{align}  
\end{subequations}
respectively.
These two estimates naturally lead to a case distinction according to whether \(z\in\Xi\) or \(z\notin\Xi\).

\textbf{Step 1:} For \(z\in\Xi\), \eqref{firstBS} immediately yields
\begin{equation}\label{firstin}
    \left(\mathrm{spec}( P + V)\cap\Xi\right)
\subset \Bigl\{ z \in \mathbb{C} : |z|^{1-\frac1m} (1+|z|)^{-\frac{2\sigma(q)}{m}} \le C \,\|V\|_{L^q(M)} \Bigr\}.
\end{equation}

\textbf{Step 2:} For \(z\notin\Xi\), we distinguish between the cases \(\Re z\le 0\) and \(\Re z>0\).

\textit{Step 2.1:} We first consider the case $\Re(z)\le 0$. Since the maps \(\lambda\mapsto |(\lambda\pm i)^m|\) are increasing for \(\lambda\ge \cot(\pi/m)\), and $$\Gamma\cap i\R=\{\pm i\csc^m(\pi/(2m))\},$$ 
we have by \eqref{secondBS}: 
\begin{equation}\label{secondin}
    \left(\mathrm{spec}( P + V)\cap\Xi^c\cap\left\{z\in\C:\Re(z)\le 0\right\}\right)
\subset D(\lambda^m_0,Cr_0),
\end{equation}
where we recall that $\lambda_0=0$.

\textit{Step 2.2:} We now consider the case \(\Re(z)>0\).
Then, considering the eigenvalues without multiplicity, there exists a unique
\(k\in\mathbb N\) such that
\[
    \Re z\in[\lambda_{k-1}^m,\lambda_k^m).
\]
and hence \eqref{secondBS} implies that \(d(z)\lesssim r_k\). Moreover, using
$d(z)=\min\bigl\{|z-\lambda_{k-1}^m|,\ |z-\lambda_k^m|\bigr\}$
together with Weyl's asymptotic law,
$\lambda_n/\lambda_{n-1}\to 1$ as $n\to\infty,$
we obtain
\begin{equation}\label{thirdin}
\left(\operatorname{spec}\bigl( P+V\bigr)\cap\Xi^c\cap\{z\in\mathbb C:\Re(z)\in[\lambda_{k-1}^m,\lambda_k^m)\}\right)
\subset D(\lambda_{k-1}^m,Cr_{k-1})\cup D(\lambda_k^m,Cr_k),
\end{equation}
for all \(k\in\mathbb N\).

Upon setting \(\nu(p')=\sigma(q)\), the inclusions \eqref{firstin}, \eqref{secondin}, and \eqref{thirdin} combine to give \eqref{result}, and thus Theorem~\ref{psithm1} follows.

\end{proof}

\section{Optimality under Zoll-type spectral clustering: proof of Theorem~\ref{thmopt}}
\label{secopt}

The sharpness argument is based on spectral clustering for Zoll manifolds
and, more generally, for elliptic operators with periodic
bicharacteristic flow; see
\cite{Besse1978,Weinstein1977Clusters,DuistermaatGuillemin1975Periodic,ColinDeVerdiere1979Periodic}.
On the round sphere, the relevant spectral-cluster estimates are sharp
\cite{Sogge1986,SOGGE1988,CS}. Building on this structure, Cuenin
\cite{CJC} established the optimality, up to multiplicative constants, of
spectral-inclusion radii for complex Schrödinger operators on spheres and
Zoll manifolds. We adapt this mechanism to positive elliptic
pseudodifferential operators of order \(m\) satisfying the abstract
Zoll-type clustering assumptions of Theorem~\ref{thmopt}.

\begin{proof} The proof proceeds by constructing a potential \(V_k\) for which the
resonant part of the associated Birman--Schwinger family has a
distinguished characteristic value
\[
    z_{\mathrm r,0}
    =
    \Lambda_k^m
    +
    \eta_k e^{i\theta}k^{2\sigma(q)}.
\]
We then show that the nonresonant contribution is sufficiently small on a
suitable contour surrounding \(z_{\mathrm r,0}\). The operator-valued
Rouch\'e theorem of Gohberg and Sigal \cite[Theorem~2.2]{GS71} implies
that the full Birman--Schwinger family has a characteristic value \(z_k\)
enclosed by this contour. By the Birman--Schwinger principle, \(z_k\) is
an eigenvalue of \(P+V_k\) satisfying
\[
    z_k
    =
    \Lambda_k^m
    +
    \eta_k e^{i\theta}k^{2\sigma(q)}
    +
    O\left(
        \eta_k^2 k^{4\sigma(q)+1-m}
        +
        k^{m-1}\delta_k
    \right).
\]

\subsection{\texorpdfstring{The \(L^\infty\)-case}{The L-infinity case}}
We first treat the endpoint case \(q=\infty\). Then
\(p=p'=2\) and \(\sigma(\infty)=0\). By
\eqref{eq:sharp-cluster-lower}, \(\Pi_k\neq0\), so we may choose
\[
    \tau_k\in\operatorname{spec}(A)\cap I_k.
\]
The clustering assumption \eqref{eq:zoll-type-clustering} and the mean
value theorem give
\[
    \tau_k^m
    =
    \Lambda_k^m+O(k^{m-1}\delta_k).
\]
Taking the constant potential
\[
    V_k=\eta_k e^{i\theta},
\]
we obtain
\[
    \operatorname{spec}(P+V_k)
    =
    \operatorname{spec}(P)+\eta_k e^{i\theta}.
\]
Consequently,
\[
    z_k:=\tau_k^m+\eta_k e^{i\theta}
    =
    \Lambda_k^m+\eta_k e^{i\theta}
    +O(k^{m-1}\delta_k)
\]
is an eigenvalue of \(P+V_k\), and
\(\|V_k\|_{L^\infty(M)}=\eta_k\). This proves the assertion for
\(q=\infty\). We henceforth assume \(1<q<\infty\).

\subsection{The two error scales}
\label{sec:scales}

We introduce two quantities:
\begin{equation}
   S_k:=\eta_k k^{2\sigma(q)},
   \qquad
   W_k:=C_W k^{m-1}\delta_k,
   \label{eq:S-W}
\end{equation}
where $C_W>0$ will later be chosen sufficiently large.

The quantity $S_k$ is the desired eigenvalue displacement. The quantity
$W_k$ is the error caused by the fact that the cluster is not concentrated
at the single frequency $\Lambda_k$.

Define
\begin{equation}
   \alpha_k:=\frac{S_k}{k^{m-1}}
   =\eta_k k^{2\sigma(q)+1-m},
   \qquad
   \beta_k:=\frac{k^{m-1}\delta_k}{S_k}.
   \label{eq:alpha-beta}
\end{equation}
This implies the following relations:
\begin{equation}
\label{eq:error-scale-relations}
   \varepsilon_k=\alpha_k+\beta_k,
   \qquad
   \frac{W_k}{S_k}=C_W\beta_k.
\end{equation}
The convergence $\varepsilon_k\to0$ implies separately that
\begin{equation}
\label{eq:small-consequences}
   \alpha_k\to0,
   \qquad \beta_k\to0,
   \qquad W_k=o(S_k),
   \qquad
   \alpha_k\beta_k
   =\delta_k.
\end{equation}
It follows in particular that $\delta_k\to0$.

Finally, since $\Lambda_k^m\asymp k^m$,
\begin{equation}
   \frac{S_k}{\Lambda_k^m}
   \asymp \eta_k k^{2\sigma(q)-m}
   =\frac{\alpha_k}{k}\longrightarrow0.
   \label{eq:S-small-energy}
\end{equation}
Thus the desired displacement is small compared with the base energy
$\Lambda_k^m$, but large compared with the cluster width $W_k$.

\subsection{Choice of the weight}
We construct a normalized weight \(\varphi_k\) from an almost extremizer of
the spectral-cluster estimate and show that the resulting weighted projection
\(T_k=\varphi_k\Pi_k\varphi_k\) has a largest eigenvalue of the required size
\(k^{2\sigma(q)}\).

Since \(\Pi_k\) is an orthogonal projection,
\[
    \|\Pi_k\|_{L^p(M)\to L^{p'}(M)}
    =
    \|\Pi_k\|_{L^p(M)\to L^2(M)}^2.
\]
Consequently, \eqref{eq:sharp-cluster-lower} implies
\[
    \|\Pi_k\|_{L^p(M)\to L^2(M)}
    \gtrsim
    k^{\sigma(q)}.
\]
We may therefore choose \(f_k\in L^p(M)\), with
\(\|f_k\|_{L^p(M)}=1\), such that
\[
    \|\Pi_k f_k\|_{L^2(M)}
    \gtrsim
    k^{\sigma(q)}.
\]
Since
\[
    \frac1p=\frac12+\frac1{2q},
\]
we can factor \(f_k\) as
\[
    f_k=\varphi_k g_k,
\]
where
\[
    \varphi_k:=|f_k|^{p/(2q)},
    \qquad
    g_k:=
    \begin{cases}
        f_k/\varphi_k,&\varphi_k\neq0,\\
        0,&\varphi_k=0.
    \end{cases}
\]
The choice of the exponents gives
\[
    \|\varphi_k\|_{L^{2q}(M)}=1,
    \qquad
    \|g_k\|_{L^2(M)}=1.
\]
It follows from the construction that
\begin{equation}
\label{eq:weighted-projector-lower}
    \|\Pi_k\varphi_k\|_{L^2(M)\to L^2(M)}
    \geq
    \|\Pi_k\varphi_k g_k\|_{L^2(M)}
    \gtrsim
    k^{\sigma(q)}.
\end{equation}
Set
\[
    T_k:=\varphi_k\Pi_k\varphi_k,
\]
which is a positive finite-rank operator on \(L^2(M)\). By
\eqref{eq:weighted-projector-lower},
\begin{equation}
\label{eq:Tk-lower}
    \|T_k\|_{L^2(M)\to L^2(M)}
    =
    \|\Pi_k\varphi_k\|_{L^2(M)\to L^2(M)}^2
    \gtrsim
    k^{2\sigma(q)}.
\end{equation}
Thus, if
\[
    a_0(k)\geq a_1(k)\geq\cdots\geq0
\]
are the eigenvalues of \(T_k\), counted with multiplicity, then
\[
\begin{aligned}
a_0(k)
&=
\|\varphi_k \Pi_k \varphi_k\|_{L^2(M)\to L^2(M)}
=
\|\Pi_k \varphi_k\|_{L^2(M)\to L^2(M)}^{2}
\\
&\leq
\|\Pi_k\|_{L^p(M)\to L^2(M)}^{2}
\|\varphi_k\|_{L^{2q}(M)}^{2}
\lesssim
k^{2\sigma(q)},
\end{aligned}
\]
where the final inequality follows from the unit spectral-cluster estimate
\eqref{Sogge2} and \(\|\varphi_k\|_{L^{2q}(M)}=1\). Together with
\eqref{eq:Tk-lower}, this implies
\begin{equation}
\label{eq:a0-size}
    a_0(k)\asymp k^{2\sigma(q)}.
\end{equation}

\subsection{Birman--Schwinger reduction and resonant splitting}
In this step, we use the weight \(\varphi_k\) to define the potential and convert
the eigenvalue problem into a Birman--Schwinger characteristic-value problem.

We seek the potential in the form
\[
    V_k=\zeta\varphi_k^2,
\]
where \(\zeta\) will be chosen below.
For \(z\notin\operatorname{spec}(P)\), define
\[
    \mathcal A(z,\zeta)
    :=
    I+\zeta\varphi_k(P-z)^{-1}\varphi_k.
\]
By the Birman--Schwinger principle, \(z\) is an eigenvalue of
\(P+\zeta\varphi_k^2\) if and only if \(\mathcal A(z,\zeta)\) is not
invertible.

We decompose the weighted resolvent into a resonant and a nonresonant part:
\begin{equation}
\label{eq:resonant-splitting}
    \varphi_k(P-z)^{-1}\varphi_k
    =
    K_{\mathrm r}(z)+K_{\mathrm{nr}}(z),
\end{equation}
where
\begin{equation}
\label{eq:abstract-resonant-part}
    K_{\mathrm r}(z)
    :=
    \frac{T_k}{\Lambda_k^m-z}.
\end{equation}
The corresponding resonant operator-valued function is
\[
    \mathcal A_{\mathrm r}(z,\zeta)
    :=
    I+\zeta K_{\mathrm r}(z).
\]
Its characteristic values are
\begin{equation}
\label{eq:resonant-characteristic-values}
    z_{\mathrm r,j}(\zeta)
    =
    \Lambda_k^m+\zeta a_j(k).
\end{equation}
Choose
\begin{equation}
\label{eq:zeta0-choice}
    \zeta_{0,k}
    :=
    \frac{
        \eta_k e^{i\theta}k^{2\sigma(q)}
    }{
        a_0(k)
    }.
\end{equation}
By \eqref{eq:a0-size},
\begin{equation}
\label{eq:zeta0-size}
    |\zeta_{0,k}|\asymp\eta_k,
\end{equation}
and the distinguished characteristic value associated with \(a_0(k)\) is
\begin{equation}
\label{eq:resonant-zero}
    z_{\mathrm r,0}
    :=
    z_{\mathrm r,0}(\zeta_{0,k})
    =
    \Lambda_k^m
    +
     e^{i\theta}S_k.
\end{equation}

From now on, set
\[
    V_k:=\zeta_{0,k}\varphi_k^2,
    \qquad
    \mathcal A_k(z):=\mathcal A(z,\zeta_{0,k}),
    \qquad
    \mathcal A_{k,\mathrm r}(z)
    :=\mathcal A_{\mathrm r}(z,\zeta_{0,k}).
\]
Moreover,
\begin{equation}
\label{eq:potential-norm}
    \|V_k\|_{L^q(M)}
    =
    |\zeta_{0,k}|
    \|\varphi_k^2\|_{L^q(M)}
    =
    |\zeta_{0,k}|
    \asymp\eta_k.
\end{equation}

If \(v_j\) is an eigenvector of \(T_k\) corresponding to \(a_j(k)\), then
\begin{equation}
\label{eq:resonant-scalar-action}
    \mathcal A_{k,\mathrm r}(z)v_j
    =
    \frac{z-z_{\mathrm r,j}(\zeta_{0,k})}
         {z-\Lambda_k^m}v_j.
\end{equation}

\subsection{The local change of spectral parameter}
We pass from the order-\(m\) energy parameter \(z\) to the
first-order frequency parameter \(z^{1/m}\).

We shall work in the spectral window
\[
    |z-\Lambda_k^m|\leq 2S_k.
\]
Equation \eqref{eq:S-small-energy} implies
\[
    |z-\Lambda_k^m|=o(\Lambda_k^m).
\]
We may therefore use the holomorphic branch of \(z^{1/m}\) defined in a
neighborhood of the positive number \(\Lambda_k^m\). The fundamental theorem
of calculus along the line segment joining \(\Lambda_k^m\) and \(z\) gives
\begin{align}
\label{eq:mth-root-estimate}
    |z^{1/m}-\Lambda_k|
    &\lesssim
    \Lambda_k^{1-m}|z-\Lambda_k^m|
    \notag\\
    &\lesssim
    \eta_k k^{2\sigma(q)+1-m}
    =\alpha_k.
\end{align}
Thus
\[
    |z^{1/m}-\Lambda_k|=o(1).
\]
Since the intervals \(I_k\) are separated, the only spectral cluster of
\(A\) meeting a fixed neighborhood of \(z^{1/m}\) is the \(k\)-th one.

\subsection{Estimate of the nonresonant part}
In this step, we derive a uniform bound for the nonresonant part
\(K_{\mathrm{nr}}(z)\) of the weighted resolvent in the spectral region
considered above.

We first quantify the width of the \(k\)-th spectral cluster on the
\(P=A^m\) energy scale. Choose the constant \(C_W\) in \eqref{eq:S-W}
sufficiently large. If
\(\tau\in\operatorname{spec}(A)\cap I_k\), then
\eqref{eq:zoll-type-clustering} gives
\[
    |\tau-\Lambda_k|\leq\delta_k.
\]
Applying the mean value theorem to \(t\mapsto t^m\), there exists a point
\(\xi\) between \(\tau\) and \(\Lambda_k\) such that
\[
    \tau^m-\Lambda_k^m
    =
    m\xi^{m-1}(\tau-\Lambda_k).
\]
Since \(\tau\in I_k\), we have \(\xi\asymp\Lambda_k\asymp k\), and hence
\begin{equation}
\label{eq:energy-cluster-width}
    |\tau^m-\Lambda_k^m|
    \lesssim
    k^{m-1}\delta_k
    \leq
    W_k.
\end{equation}

We decompose the nonresonant part as
\[
    K_{\mathrm{nr}}(z)
    =
    K_{\mathrm{nr}}^{(1)}(z)
    +
    K_{\mathrm{nr}}^{(2)}(z),
\]
where
\begin{align}
    K_{\mathrm{nr}}^{(1)}(z)
    &:=
    \varphi_k\Pi_k
    \left(
        (P-z)^{-1}
        -
        \frac1{\Lambda_k^m-z}
    \right)
    \Pi_k\varphi_k,
    \label{eq:intracluster-remainder}
    \\
    K_{\mathrm{nr}}^{(2)}(z)
    &:=
    \varphi_k(1-\Pi_k)(P-z)^{-1}(1-\Pi_k)\varphi_k.
    \label{eq:outercluster-remainder}
\end{align}
The first term measures the error made by replacing the exact resolvent on
the \(k\)-th cluster by the constant denominator
\((\Lambda_k^m-z)^{-1}\), whereas the second term contains all spectral
clusters different from the \(k\)-th one.

\subsubsection{Estimate of \(K_{\mathrm{nr}}^{(1)}\)}
Let
\[
    d_k(z):=|\Lambda_k^m-z|.
\]
Suppose that
\[
    d_k(z)\geq 2W_k.
\]
For every
\(\tau\in\operatorname{spec}(A)\cap I_k\), equation
\eqref{eq:energy-cluster-width} gives
\[
\begin{aligned}
    |\tau^m-z|
    &\geq
    |\Lambda_k^m-z|
    -
    |\tau^m-\Lambda_k^m|
    \\
    &\geq
    d_k(z)-W_k
    \geq
    \frac{d_k(z)}{2},
\end{aligned}
\]
while
\[
    |\tau^m-z|
    \leq
    d_k(z)+W_k
    \leq
    \frac{3d_k(z)}{2}.
\]
Thus
\[
    |\tau^m-z|
    \asymp
    d_k(z)
\]
uniformly on the \(k\)-th cluster. Consequently,
\begin{align*}
    \left|
        \frac1{\tau^m-z}
        -
        \frac1{\Lambda_k^m-z}
    \right|
    &=
    \frac{
        |\tau^m-\Lambda_k^m|
    }{
        |\tau^m-z|\,|\Lambda_k^m-z|
    }
    \\
    &\lesssim
    \frac{W_k}{d_k(z)^2}.
\end{align*}
Since
\[
    \|\varphi_k\Pi_k\|_{L^2\to L^2}^2
    =
    \|\varphi_k\Pi_k\varphi_k\|_{L^2\to L^2}
    \lesssim
    k^{2\sigma(q)},
\]
the spectral theorem therefore yields
\begin{equation}
\label{eq:intracluster-bound}
    \|K_{\mathrm{nr}}^{(1)}(z)\|_{L^2(M)\to L^2(M)}
    \lesssim
    \frac{W_k k^{2\sigma(q)}}{d_k(z)^2}.
\end{equation}

\subsubsection{Estimate of \(K_{\mathrm{nr}}^{(2)}\)}
We next estimate the contribution of all clusters different from the
\(k\)-th one. By \eqref{eq:mth-root-estimate},
\[
    z^{1/m}
    =
    \Lambda_k+o(1),
\]
so the \(k\)-th cluster is the only spectral cluster that can be resonant
with \(z\). Repeating the proof of \eqref{secondequationSogge}, with
\(I_k\) as the deleted spectral window, gives
\begin{equation}
\label{eq:complementary-resolvent-sharpness}
    \left\|
        (1-\Pi_k)(P-z)^{-1}(1-\Pi_k)
    \right\|_{L^p(M)\to L^{p'}(M)}
    \lesssim
    |z|^{\frac{2\sigma(q)+1}{m}-1},
\end{equation}
throughout the region \(|z-\Lambda_k^m|\leq2S_k\).

For completeness, the denominator estimates underlying this bound are as
follows. Since \(z^{1/m}=\Lambda_k+o(1)\), and hence
\(|z|\asymp\Lambda_k^m\asymp k^m\),
we have, uniformly for \(\tau\notin I_k\),
\[
    |\tau^m-z|
    \gtrsim
    \begin{cases}
        k^m,
        & \tau\leq\Lambda_k/2,\\[0.4ex]
        k^{m-1}\bigl(1+|\tau-\Lambda_k|\bigr),
        & \tau\asymp k,\\[0.4ex]
        \tau^m,
        & \tau\geq2\Lambda_k.
    \end{cases}
\]

Moreover, equation \eqref{eq:complementary-resolvent-sharpness} becomes
\begin{equation}
\label{eq:nonresonant-resolvent-bound}
    \left\|
        (1-\Pi_k)(P-z)^{-1}(1-\Pi_k)
    \right\|_{L^p(M)\to L^{p'}(M)}
    \lesssim
    k^{2\sigma(q)+1-m}.
\end{equation}
Hence \eqref{eq:nonresonant-resolvent-bound}, together with Hölder's
inequality applied on both sides of the resolvent in
\eqref{eq:outercluster-remainder}, yields
\begin{equation}
\label{eq:outercluster-bound}
    \|K_{\mathrm{nr}}^{(2)}(z)\|_{L^2(M)\to L^2(M)}
    \lesssim
    k^{2\sigma(q)+1-m}.
\end{equation}
The factor \(k^{1-m}\) reflects the conversion from the first-order
frequency scale of \(A\) to the order-\(m\) energy scale of \(P=A^m\).

Combining \eqref{eq:intracluster-bound} and
\eqref{eq:outercluster-bound}, we conclude that
\begin{equation}
\label{eq:full-nonresonant-bound}
    \|K_{\mathrm{nr}}(z)\|_{L^2(M)\to L^2(M)}
    \lesssim
    \frac{
        W_k k^{2\sigma(q)}
    }{
        |\Lambda_k^m-z|^2
    }
    +
    k^{2\sigma(q)+1-m},
\end{equation}
whenever
\begin{equation}
\label{eq:nonresonant-bound-region}
    2W_k\leq|z-\Lambda_k^m|\leq2S_k.
\end{equation}

\subsection{Separation of the resonant characteristic values}
We use a Schatten bound to show that only finitely many resonant
characteristic values can lie close to the target value. This makes it
possible to choose a contour on which the inverse of the resonant operator
is uniformly controlled.

By the weighted Schatten spectral-cluster estimate of Frank and Sabin
\cite[Theorem~2 and Remarks~6 and~12]{FS17}, applied to the first-order
operator \(A=P^{1/m}\), we have
\begin{equation}
\label{eq:cluster-Schatten-assumption}
    \left\|
        \varphi\Pi_k\varphi
    \right\|_{\mathfrak S^{\beta(q)}(L^2(M))}
    \lesssim
    k^{2\sigma(q)}
    \|\varphi\|_{L^{2q}(M)}^2
\end{equation}
for every non-negative \(\varphi\in L^{2q}(M)\), where
\[
    \beta(q)
    :=
    \begin{cases}
    \displaystyle
    \frac{(d-1)q}{d-q},
    &
    \displaystyle
    1<q\leq\frac{d+1}{2},
    \\[1.4ex]
    \displaystyle
    2q,
    &
    \displaystyle
    \frac{d+1}{2}\leq q<\infty.
    \end{cases}
\]
Since $\|\varphi_k\|_{L^{2q}(M)}=1$,
equation
\eqref{eq:cluster-Schatten-assumption} yields
\[
    \left(
        \sum_{j=0}^{\infty}
        a_j(k)^{\beta(q)}
    \right)^{1/\beta(q)}
    =
    \|T_k\|_{\mathfrak S^{\beta(q)}(L^2(M))}
    \lesssim
    k^{2\sigma(q)}.
\]
Together with \(a_0(k)\asymp k^{2\sigma(q)}\), this implies
\[
    a_j(k)
    \lesssim
    (j+1)^{-1/\beta(q)}k^{2\sigma(q)}
    \lesssim
    (j+1)^{-1/\beta(q)}a_0(k),
    \qquad j\geq0.
\]
Define
\[
    t_j(k):=\frac{a_j(k)}{a_0(k)}.
\]
Then
\[
    0\leq t_j(k)\lesssim(j+1)^{-1/\beta(q)}.
\]
We may therefore choose an integer \(J\), independent of \(k\), such that
\begin{equation}
\label{eq:far-resonant-values}
    t_j(k)\leq\frac12,
    \qquad j\geq J,
\end{equation}
for every sufficiently large \(k\). Thus only the at most \(J\) indices
\(j<J\) can generate resonant characteristic values lying within distance
\(S_k/2\) of the target value.

\subsection{Choice of the contour with the crucial large fixed factor}

Equations \eqref{eq:S-W}, \eqref{eq:resonant-characteristic-values}, and
\eqref{eq:zeta0-choice} imply
\begin{equation}
   z_{\mathrm r,j}
   :=z_{\mathrm r,j}(\zeta_{0,k})
   =\Lambda_k^m+e^{i\theta}S_kt_j(k).
   \label{eq:resonant-ray}
\end{equation}
Thus all resonant values lie on the ray emanating from $\Lambda_k^m$ in the
direction $e^{i\theta}$, between the base point and the target value. Their
distance from the target value is the real number
\begin{equation}
   |z_{\mathrm r,0}-z_{\mathrm r,j}|
   =S_k(1-t_j(k)).
   \label{eq:radial-distances}
\end{equation}

Let $C_0>1$ be a fixed constant for the moment, and set
\begin{equation}
   \widetilde\varepsilon_k:=C_0\varepsilon_k.
   \label{eq:epsilon-tilde}
\end{equation}
Since $\varepsilon_k\to0$, we also have
$\widetilde\varepsilon_k\to0$.

Consider the interval of radii
\[
   \mathcal I_k
   :=[\widetilde\varepsilon_kS_k,
      (4J+4)\widetilde\varepsilon_kS_k].
\]
Divide $\mathcal I_k$ into $2J+1$ consecutive subintervals of equal length.
The at most $J$ numbers
\[
   S_k(1-t_j(k)),\qquad j<J,
\]
are the forbidden radii. Each forbidden radius can meet the closures of at
most two subintervals. Hence the pigeonhole principle provides a subinterval
whose closure contains no forbidden radius. Choose $r_k$ to be its midpoint.
Then there is a constant $c_J>0$, independent of $k$, such that
\begin{equation}
\widetilde\varepsilon_kS_k\leq r_k
\leq(4J+4)\widetilde\varepsilon_kS_k
\label{eq:radius-two-sided}
\end{equation}
and
\begin{equation}
   \operatorname{dist}\!\left(r_k,
      \{S_k(1-t_j(k)):j<J\}\right)
   \geq c_J\widetilde\varepsilon_kS_k.
   \label{eq:radius-separation-near}
\end{equation}

For $j\geq J$, the forbidden radius is at least $S_k/2$ by
\eqref{eq:far-resonant-values}. On the other hand, since
$\widetilde\varepsilon_k\to0$, we have $r_k\leq S_k/4$ for all sufficiently large
$k$. Hence the circle
\begin{equation}
   \Gamma_k:=\{z\in\mathbb C:|z-z_{\mathrm r,0}|=r_k\}
   \label{eq:Gamma}
\end{equation}
is separated from all resonant values:
\begin{equation}
   \operatorname{dist}\!\left(z,\{z_{\mathrm r,j}:j\geq0\}\right)
   \geq c\widetilde\varepsilon_kS_k,
   \qquad z\in\Gamma_k.
   \label{eq:all-separation}
\end{equation}

\subsection{The inverse norm of the resonant model}

The scalar formula \eqref{eq:resonant-scalar-action} gives, on the eigenspace
corresponding to $a_j(k)$,
\[
   \left\|
       \left.\mathcal A_{k,\mathrm r}(z)^{-1}
       \right|_{\ker(T_k-a_j(k))}
   \right\|
   =\frac{|z-\Lambda_k^m|}{|z-z_{\mathrm r,j}|}.
\]
On $\ker T_k$, the operator $\mathcal A_{k,\mathrm r}(z)$ is the identity.
For $z\in\Gamma_k$,
\[
   |z-\Lambda_k^m|
   \leq|z-z_{\mathrm r,0}|
       +|z_{\mathrm r,0}-\Lambda_k^m|
   =r_k+S_k\leq2S_k
\]
for all sufficiently large $k$. Together with \eqref{eq:all-separation},
this yields
\begin{equation}
   \sup_{z\in\Gamma_k}
   \left\|\mathcal A_{k,\mathrm r}(z)^{-1}\right\|_{L^2(M)\to L^2(M)}
   \leq\frac{C}{\widetilde\varepsilon_k}
   =\frac{C}{C_0\varepsilon_k}.
   \label{eq:Ar-inverse-final}
\end{equation}

The fixed factor $C_0$ will allow us to make the Rouch\'e perturbation
strictly smaller than one.

\subsection{Application of the operator-valued Rouch\'e theorem}
\label{sec:Rouche}

\subsubsection{The Gohberg--Sigal theorem}
We use the following operator-valued version of Rouch\'e's theorem
\cite[Theorem~2.2]{GS71}.

\begin{theorem}[Gohberg--Sigal]
\label{thm:GS}
Let $D\subset\mathbb C$ be a bounded domain with piecewise smooth,
positively oriented boundary. Let $F$ and $G$ be holomorphic
operator-valued functions on a neighborhood of $\overline D$, taking values
in the Fredholm operators of index zero. Suppose that $G$ is invertible on
$\partial D$ and that
\[
   \sup_{z\in\partial D}
   \left\|G(z)^{-1}(F(z)-G(z))\right\|<1.
\]
Then $F$ and $G$ have the same number of characteristic values in $D$,
counted with algebraic multiplicity.
\end{theorem}

We apply the theorem with
\[
   F=\mathcal A_k,
   \qquad
   G=\mathcal A_{k,\mathrm r},
   \qquad
   \mathfrak D_k:=D(z_{\mathrm r,0},r_k),
   \qquad
   \partial\mathfrak D_k=\Gamma_k.
\]

\subsubsection{The contour lies in the resolvent set}
By \eqref{eq:radius-two-sided} and
$\widetilde\varepsilon_k=C_0\varepsilon_k\to0$, we have
$r_k=o(S_k)$. Hence, for every $z\in\overline{\mathfrak D_k}$,
\begin{equation}
\label{eq:on-disk-scale}
    |z-\Lambda_k^m|\asymp S_k.
\end{equation}
In particular, the pole $\Lambda_k^m$ of the resonant model lies outside
$\overline{\mathfrak D_k}$.

We claim that
\begin{equation}
\label{eq:disk-in-resolvent-set}
    \overline{\mathfrak D_k}\cap\operatorname{spec}(P)=\varnothing
\end{equation}
for all sufficiently large $k$. Indeed, suppose that
$\tau^m\in\operatorname{spec}(P)$ belongs to
$\overline{\mathfrak D_k}$. Applying \eqref{eq:mth-root-estimate} with
$z=\tau^m$ gives $\tau=\Lambda_k+o(1)$, and hence $\tau\in I_k$ for all
sufficiently large $k$. By \eqref{eq:energy-cluster-width},
\[
    |\tau^m-\Lambda_k^m|\leq W_k=o(S_k),
\]
which contradicts \eqref{eq:on-disk-scale}. This proves
\eqref{eq:disk-in-resolvent-set}.

It follows that $\mathcal A_k$ is a holomorphic compact perturbation of the
identity on a neighborhood of $\overline{\mathfrak D_k}$, whereas
$\mathcal A_{k,\mathrm r}$ is a holomorphic finite-rank perturbation of the
identity there. Both families are therefore Fredholm of index zero. Moreover,
\eqref{eq:all-separation} shows that $\mathcal A_{k,\mathrm r}$ is invertible
on $\Gamma_k$.

\subsubsection{Verification of the Rouch\'e inequality}
For $z\in\Gamma_k$, equations \eqref{eq:on-disk-scale} and
\eqref{eq:small-consequences} show that the conditions in
\eqref{eq:nonresonant-bound-region} hold for all sufficiently large $k$.
Using \eqref{eq:resonant-splitting}, \eqref{eq:zeta0-size},
\eqref{eq:full-nonresonant-bound}, \eqref{eq:Ar-inverse-final}, and
\eqref{eq:alpha-beta}--\eqref{eq:error-scale-relations}, we obtain
\begin{align}
&\sup_{z\in\Gamma_k}
\left\|
    \mathcal A_{k,\mathrm r}(z)^{-1}
    \bigl(\mathcal A_k(z)-\mathcal A_{k,\mathrm r}(z)\bigr)
\right\|_{L^2(M)\to L^2(M)}
\notag\\
&\quad=
\sup_{z\in\Gamma_k}
\left\|
    \mathcal A_{k,\mathrm r}(z)^{-1}
    \zeta_{0,k}K_{\mathrm{nr}}(z)
\right\|_{L^2(M)\to L^2(M)}
\notag\\
&\quad\lesssim
\frac{1}{C_0\varepsilon_k}\,
\eta_k
\left(
    \frac{W_k k^{2\sigma(q)}}{S_k^2}
    +k^{2\sigma(q)+1-m}
\right)
\notag\\
&\quad=
\frac{1}{C_0\varepsilon_k}
\left(
    \frac{W_k}{S_k}
    +\eta_k k^{2\sigma(q)+1-m}
\right)
\leq
\frac{C}{C_0}.
\label{eq:Rouche-small}
\end{align}
Choose the fixed constant $C_0>2C$. Then the right-hand side of
\eqref{eq:Rouche-small} is strictly smaller than $1$. Since $C_0$ is
independent of $k$, this choice does not alter the asymptotic order of the
contour radius.

\subsection{Conclusion of the proof}
All hypotheses of Theorem~\ref{thm:GS} are now verified. The resonant model
has the characteristic value $z_{\mathrm r,0}$ at the center of
$\mathfrak D_k$. Theorem~\ref{thm:GS} therefore implies that
$\mathcal A_k$ has at least one characteristic value $z_k$ in
$\mathfrak D_k$. By \eqref{eq:disk-in-resolvent-set}, the
Birman--Schwinger principle applies and gives
\[
    z_k\in\operatorname{spec}(P+V_k).
\]
Equations \eqref{eq:resonant-zero}, \eqref{eq:radius-two-sided},
\eqref{eq:epsilon-tilde}, and \eqref{eq:sharpness-smallness} yield
\begin{align*}
    |z_k-z_{\mathrm r,0}|
    &\leq r_k
    \lesssim
    \varepsilon_k S_k
    =
    \varepsilon_k\eta_k k^{2\sigma(q)}\\
    &=
    \eta_k^2k^{4\sigma(q)+1-m}
    +k^{m-1}\delta_k.
\end{align*}
Consequently,
\[
    z_k
    =
    \Lambda_k^m
    +\eta_ke^{i\theta}k^{2\sigma(q)}
    +O\!\left(
        \eta_k^2k^{4\sigma(q)+1-m}
        +k^{m-1}\delta_k
    \right),
\]
which is \eqref{eq:abstract-sharpness-conclusion}. Together with
\eqref{eq:potential-norm}, this completes the proof of
Theorem~\ref{thmopt}.
\end{proof}

\section{Optimality of fractional Schrödinger operator on Zoll manifolds}\label{sec: optimality laplace}

\subsection{Proof of Corollary \ref{cor:sharpness-sphere}}

\begin{proof}
For the round sphere,
\[
    \Lambda_k=\sqrt{k(k+d-1)}
\]
is an exact eigenfrequency of \(\sqrt{-\Delta_{\mathbb S^d}}\), so
\[
    \delta_k=0.
\]
The sharp projector bound follows from the sharpness of Sogge's spectral
cluster estimates, see e.g. \cite{CS}. The condition
\[
    \eta k^{2\sigma(q)+1-m}\ll1
\]
is equivalent, up to constants, to
\[
    k\gg
    1+\eta^{1/(m-1-2\sigma(q))}.
\]
The conclusion now follows from Theorem~\ref{thmopt}.
\end{proof}

\subsection{Proof of Corollary \ref{cor:sharpness-zoll}}
\begin{proof}
Weinstein's spectral clustering theorem
\cite{Weinstein1977Clusters} gives
\[
    \Lambda_k=k+\alpha,
    \qquad
    \delta_k=O(k^{-1}).
\]
Hence
\[
    k^{m-1}\delta_k=O(k^{m-2}).
\]
The two terms in \eqref{eq:sharpness-smallness} become
\[
    \eta k^{2\sigma(q)+1-m}
\]
and
\[
    \frac{k^{m-2}}{\eta k^{2\sigma(q)}}
    =
    \frac1{\eta k^{2\sigma(q)+2-m}},
\]
respectively. Thus \eqref{eq:zoll-two-conditions} is precisely the condition
that \(\varepsilon_k=o(1)\).

The sharp lower cluster bound follows from the cluster multiplicity
\[
    \operatorname{rank}\Pi_k\asymp k^{d-1},
\]
the local Weyl law, and the standard \(TT^*\) argument; see
\cite[Section~3.2]{CJC}. Theorem~\ref{thmopt} now gives
\eqref{eq:zoll-sharpness-final}.
\end{proof}

\vspace{0.4cm}

\section*{Acknowledgments}
This research was funded in whole, or in part, by the Austrian Science Fund (FWF)  [Grant-DOI 
10.55776/P35322, 10.55776/PAT5120424, and 
10.55776/PAT4632823]. 

The author is grateful to Christoph Aistleitner, Andrei Shubin, Petr Siegl, Jean-Claude Cuenin, and Konstantin Merz for many helpful discussions. He would like to express his special thanks to Jean-Claude Cuenin and Konstantin Merz for their particularly valuable support and insights.

\bibliographystyle{abbrv}
\bibliography{main}

\end{document}